\documentclass[11pt, a4paper, reqno]{amsart}
\usepackage{multicol,a4wide}
\usepackage{amsmath}
\usepackage{color}
\usepackage{cmap,mathtools}
\usepackage{cite}
\usepackage[hyperpageref]{backref}
\usepackage{xcolor}
\usepackage{hyperref}
\hypersetup{
    colorlinks=true,
    linkcolor=blue,
    filecolor=magenta,      
    urlcolor=green,
    citecolor=forestgreen,
}
\definecolor{forestgreen}{rgb}{0.13, 0.85, 0.15}
\usepackage[hyperpageref]{backref}

\newcommand{\dx}{{\mathrm{d}x}}

\newcommand{\dy}{{\mathrm{d}y}}
\newcommand{\dxy}{{\mathrm{d}x\mathrm{d}y}}
\newcommand{\D}{{\mathcal{D}}}
\makeatletter
\newcommand{\leqnomode}{\tagsleft@true\let\veqno\@@leqno}
\makeatother
\usepackage{yhmath,amsmath,amsfonts,amssymb,enumerate,amsthm,appendix,caption,lipsum,}
 \usepackage[T1]{fontenc}
 \usepackage{enumitem}
 \usepackage{mathtools, bm}
 \usepackage{cmap}
 \usepackage[hyperpageref]{backref}
 \usepackage{bigints}
 \usepackage{esint}
 \usepackage[margin=0.90in]{geometry}
\newtheorem{definition}{Definition}[section]
\newtheorem{theorem}[definition]{Theorem}
\newtheorem{example}[definition]{Example}

\newtheorem{remark}[definition]{Remark}
\newtheorem{proposition}[definition]{Proposition}
\newtheorem{lemma}[definition]{Lemma}

\numberwithin{equation}{section}
\DeclarePairedDelimiter\abs{\lvert}{\rvert}
\DeclarePairedDelimiter\norm{\lVert}{\rVert}
\makeatletter
\let\oldnorm\norm
\def\norm{\@ifstar{\oldnorm}{\oldnorm*}}

\makeatletter

\newcommand{\be} {\beta}

\newcommand{\Om} {\Omega}
\newcommand{\la} {\lambda}

\newcommand{\no} {\nonumber}
\newcommand{\noi} {\noindent}

\newcommand{\ra} {\rightarrow}

\DeclareMathAlphabet{\mathpzc}{T1}{pzc}{m}{it}

\def\w{{\widetilde w}}

\def\dx{{\,\rm d}x}

\def\Dsp{{{\mathcal D}^{s,p}(\R^N)}}

\def\sb2{{{\mathcal D}^{1,2}_0(B_1^c)}}

\def\w2r{{{ W}^{2,2}(\R^N)}}

\def\d2{{{\mathcal D}^{2,2}_0(\Om)}}

\def\C{{\mathcal C}}
\def\D{{\mathcal D}}

\def\R{{\mathbb R}}
\def\N{{\mathbb N}}
\def\F{{\mathcal F}}
\def\ep{{\epsilon}}
\def\({{\Big(}}
\def\){{\Big)}}

\def\ws2{{\F_{\frac{N}{2}}}}
\def\L2{{ L^{1,\;\infty}(\log L)^2}}
\def\dx{{\rm d}x}

\def\l2{\mathcal M\log L}

\def\c1Loc{{\C_{loc}^1}}

\def\Gagp{\iint_{\R^{2N}} \frac{|u(x)-u(y)|^p}{|x-y|^{N+sp}}\ \dxy}
\def\Gag2{\iint_{\R^{2N}} \frac{(u(x)-u(y))^2}{|x-y|^{N+sp}}\ \dxy}
\def\Gagnp{\iint_{\R^{2N}} \frac{|u_n(x)-u_n(y)|^p}{|x-y|^{N+sp}}\ \dxy}
\def\Gagn2{\iint_{\R^{2N}} \frac{(u_n(x)-u_n(y))^2}{|x-y|^{N+sp}}\ \dxy}
\usepackage{orcidlink}

\title[Fractional $p$-Laplace semipositone problems over $\R^N$]{On semipositone problems over $\R^N$ for the fractional $p$-Laplace operator}
\author[N. Biswas and R. Kumar]{Nirjan Biswas$^1$\,\orcidlink{0000-0002-3528-8388} and Rohit Kumar$^2$\,\orcidlink{0009-0001-6494-6407}} 
\address{\rm  $^1$Department of Mathematics, Indian Institute of Science Education and Research Pune \\
Dr. Homi Bhabha Road, Pune 411008, India}
\address{\rm $^2$Department of Mathematics, Indian Institute of Technology Jodhpur \\
Rajasthan 342030, India}
\email{nirjaniitm@gmail.com, nirjan.biswas@acads.iiserpune.ac.in, rohit1.iitj@gmail.com}  
\thanks{$^2$Corresponding author}
\begin{document}


\begin{abstract}
\noindent For $N \ge 1, s\in (0,1)$, and $p \in (1, \frac{N}{s})$ we find a positive solution to the following class of semipositone problems associated with the fractional $p$-Laplace operator: 
\begin{equation}\tag{SP}
          (-\Delta)_{p}^{s}u =  g(x)f_a(u) \text{ in } \mathbb{R}^N, 
\end{equation} 
where $g \in L^1(\mathbb{R}^N) \cap L^{\infty}(\mathbb{R}^N)$ is a positive function, $a>0$ is a parameter and $f_a \in \C(\mathbb{R})$ is defined as $f_a(t) = f(t)-a$ for $t \ge 0$, $f_a(t) = -a(t+1)$ for $t \in [-1, 0]$, and $f_a(t) = 0$ for $t \le -1$, where $f$ is a non-negative continuous function on $[0,\infty)$ satisfies $f(0)=0$ with subcritical and Ambrosetti-Rabinowitz type growth. Depending on the range of $a$, we obtain the existence of a mountain pass solution to (SP) in $\mathcal{D}^{s,p}(\mathbb{R}^N)$. Then, we prove mountain pass solutions are uniformly bounded with respect to $a$, over $L^r(\mathbb{R}^N)$ for every $r \in \left[\frac{Np}{N-sp}, \infty\right]$. In addition, if $p>\frac{2N}{N+2s}$, we establish that (SP) admits a non-negative mountain pass solution for each $a$ near zero. Finally, under the assumption $g(x) \leq \frac{B}{|x|^{\beta(p-1)+sp}}$ for $B>0, x \neq 0$, and $ \be \in \left(\frac{N-sp}{p-1}, \frac{N}{p-1}\right)$, we derive an explicit positive radial subsolution to (SP) and show that the non-negative solution is positive a.e. in $\mathbb{R}^N$.
\end{abstract} 
\maketitle

\section{Introduction}
\noindent In this article, for $N \ge 1, s\in (0,1)$, and $p \in (1, \frac{N}{s})$ we study the following semipositone problems associated with the fractional $p$-Laplace operator: 
\begin{equation} \label{Main problem}
          (-\Delta)_{p}^{s}u=  g(x)f_a(u) \text{ in } \R^N \tag{SP}, 
\end{equation} 
where the function $g \in L^{1}(\R^N) \cap L^{\infty}(\R^N)$ is positive, $a>0$ is a parameter and the associated function $f_a: \R \rightarrow \R$ is defined as follows:
\begin{align} \label{1.1}
    f_a(t) = \begin{cases}
        f(t)-a & \text{ if } t \geq 0,\\
        -a(t+1)      & \text{ if }t\in [-1,0],\\
        0      & \text{ if }t \leq -1,
    \end{cases}
\end{align} 
where $f$ is a non-negative continuous function on $[0,\infty)$ with $f(0)=0$.
Also, $f$ satisfies the following growth assumptions:
\begin{enumerate}[label={($\bf f{\arabic*}$)}]
    \setcounter{enumi}{0}
    \item \label{f1}  $\displaystyle \lim\limits_{t \rightarrow 0^+} \frac{f(t)}{t^{p-1}}=0,$ and $\displaystyle \lim\limits_{t \rightarrow \infty} \frac{f(t)}{t^{\gamma-1}} \leq C(f)$ for some $\gamma \in \left(p,\frac{Np}{N-sp}\right)$ and  $C(f)>0$, \vspace{0.2 cm}
    \item \label{f2} (Ambrosetti-Rabinowitz) there exist $\vartheta >p$ and $t_0>0$ such that \[0< \vartheta F(t) \leq tf(t), \quad \forall \, t>t_0, \text{ where } F(t)= \int_{0}^{t} f(\tau) \,\mathrm{d}\tau.\]
   \end{enumerate}
We consider the solution space for \eqref{Main problem} as
\begin{align*}
    \Dsp := \overline{\C_{c}^{\infty}(\R^N)}^{\norm{\cdot}_{s,p}}, \text{ where }  \|u\|_{s,p} := \bigg(\Gagp\bigg)^{\frac{1}{p}}.
\end{align*}
The homogeneous fractional Sobolev space $\Dsp$ has the following  characterization (see \cite[Theorem 3.1]{Brasco2019characterisation}):
\begin{align}\label{dsp}
   \Dsp = \left\{ u \in L^{\frac{Np}{N-sp}}(\R^N) : \norm{u}_{s,p} < \infty \right\}. 
\end{align}
The fractional $p$-Laplace operator $(-\Delta)_{p}^{s}$ is defined as
$$ (-\Delta)_{p}^{s}u(x) = 2 \lim_{\epsilon \rightarrow 0^{+}} \int_{\mathbb{R}^{N} \backslash B_{\epsilon}(x)} \frac{|u(x) - u(y)|^{p-2}(u(x)-u(y))}{|x-y|^{N+sp}}\, \dy, \quad \text{for}~x \in \mathbb{R}^{N},$$ where $B_{\epsilon}(x)$ is the ball of radius $\epsilon$ and centred at $x$.
A function $u \in \Dsp$ is called a weak solution to \eqref{Main problem} if it satisfies the following identity:
\begin{equation*}
    \begin{split}
        \iint_{\R^{2N}} \frac{\abs{ u(x)-u(y) }^{p-2}( u(x)-u(y) )(\phi(x)-\phi(y))}{|x-y|^{N+sp}}\, \dxy \\
   = \int_{\R^N} g(x)f_a(u) \phi(x) \, \dx, \quad \forall \, \phi \in \Dsp.
    \end{split}
\end{equation*}
We call \eqref{Main problem} a semipositone problem since the term  $g(x)f_a(u)$ is strictly negative on some part of the regions $\{u \le 0\}$ and $\{u > 0\}$ near $u=0$. Semipositone problems arise in mathematical biology, population dynamics, control theory, etc. (see \cite{Oruganti_Shivaji_2002,Castro_Maya2000} and the references therein). Mathematicians have used several techniques to prove the existence of positive solution for semipositone problems, which include fixed point theory, sub and super-solution methods, and variational methods. These methods help to establish conditions under which solutions exist and (possibly under some additional hypotheses) provide insights into the qualitative behaviour of these solutions. Unlike positone problems, where the strong maximum principle guarantees the positivity of a non-negative solution, semipositone problems arise when the solution lives in regions where the source term is negative.

The study of semipositone problems began from the work \cite{BSPRS82} of Brown and Shivaji while studying the bifurcation theory for the perturbed problem $-\Delta u = \lambda(u-u^3) -\epsilon$ in $\Omega$ and $u >0$ in $\Omega$, where $\lambda, \epsilon > 0 $ and $\Omega$ is a bounded domain. Subsequently, numerous authors have investigated the existence and the qualitative aspects of positive solutions to local semipositone problems across various domains. For bounded domain $\Omega$, relevant studies can be found in \cite{CRQTPAMS17, Castro_radial_1989, Castro_Hassanpour_Shivaji, Castro_Shivaji_1988, Chhetri2015, Dancer2005, Alves2019Orlicz, Castro2016, Ruyun2024, Shiqiu2021, Costa2006, Perera2020, Lee2010}. In the context where $\Omega$ is the exterior of a bounded domain, we refer  \cite{Cheetri2018, Morris2018, Lee2011, Ko2015, Dhanya2016} and the references therein.
In \cite{Alves2019}, Alves et al. first studied the semipositone problem within the entire domain $\R^N$, as described by the equation $-\Delta u = g(x)(f(u)-a)$, $u>0$ in $\mathbb{R}^N ; N \geq 3$ with $f(0) = 0$ and $a>0$. The function $f \in \mathcal{C}(\R^+)$ is locally Lipschitz, has subcritical and the Ambrosetti-Rabinowitz (A-R) type growth. The weight $g$ is positive and bounded by a radial function $P \in \C(\R^+)$, where $P$ satisfies $(a) \int_{\R^N} |x|^{2-N}P(|x|)\,\dx < \infty$, $(b)$ $P(\abs{\cdot}) \in L^1(\R^N) \cap L^\infty(\R^N)$, and $(c)$ $|x|^{N-2}\int_{\R^N} P(|y|) |x-y|^{-N+2}\,\dy \leq C$, for all $x \in \R^N \setminus \{0\}$ and for some constant $C>0$.  For $p \in (1,N)$, in \cite{Santos2023}, Santos et al. investigated the nonlinear variant $-\Delta_p u = g(x)(f(u)-a)$, $u>0$ in $\mathbb{R}^N$, where $f(0) = 0, f \in \C(\R^+)$ exhibits subcritical with A-R type growth and the weight $g \in L^1(\R^N) \cap L^\infty(\R^N)$  satisfies $g(x)<B|x|^{-\vartheta}$ for $x\neq 0$, with $\vartheta >N$ and $B>0$.  Meanwhile, the study in \cite{Nirjan2023DCDS} focused on $\Delta^2 u = g(x)(f(u)-a)$, $u>0$ in $\mathbb{R}^N; N \geq 5$ where $f(0) = 0$, $f \in \C(\R^+)$ satisfies weaker A-R type growth.

Generally, the techniques used to prove the existence of positive solution for semipositone problems defined on $\R^N$ differ from those defined on smaller domains. 
The authors in \cite{Alves2019} studied an auxiliary problem $-\Delta u = g(x)f_a(u)$ in $\mathbb{R}^N$, where $f_a \in \C(\R)$ is defined as in \eqref{1.1}.
Subsequently, their main ideas to obtain positive solution are as follows: establish uniform boundedness of the weak solutions $\{ u_a \}$ in $L^\infty(\R^N)$ (using the regularity estimate by Brezis and Kato in \cite{Brezis_Kato}),  for $a$ near zero prove uniform convergence of $\{ u_a \}$ (using the Riesz potential for the Laplace operator) to a non-negative function $\tilde{u}$ which is a weak solution of some positone problem, find the positivity of $\tilde{u}$ (applying the strong maximum principle), and finally (again using the Riesz potential along with the assumption $(c)$) obtain the positivity of $u_a$ for $a$ near zero. For $p \neq 2$, the Riesz representation for the $p$-Laplace operator is unavailable. The authors in \cite{Santos2023} considered $-\Delta_p u = g(x) f_a(u)$, with a discontinuous function $f_a$ defined as $f_a(t)=f(t) -a$ for $t \ge 0$, $f_a(t) = 0$ for $t < 0$, and used a non-smooth variational approach. A key benefit of non-smooth analysis is that the critical point of the energy functional remains non-negative despite not being a weak solution to the problem. They established the existence of a positive critical point by constructing an explicit positive radial solution to a certain non-local equation and applying the comparison principle with that solution. This positive critical point ultimately serves as weak solution. The authors in \cite{Nirjan2023DCDS} applied a similar technique as in \cite{Alves2019} to obtain a positive solution.

Few articles are available in the literature dealing with non-local semipositone problems. In this direction, for $N>2s$ and $\Omega$ bounded, the authors in \cite{DTCPAA19} studied $(-\Delta)^s u = \lambda (u^q -1) + \mu u^r$ in $\Omega$; $u=0$ in $\R^N \setminus \Omega$, where $\lambda, \mu >0$, $q \in (0,1)$ and $r \in (1, \frac{N+2s}{N-2s})$. Under certain lower bound of $\la$, they constructed a positive subsolution when $\mu=0$ and showed that there exists at least one positive solution when $0<\mu<\mu_\lambda$.
For $N>p\ge 2$ and $\Omega$ bounded, the authors in \cite{LLVJMAA23} proved that $(-\Delta)_p^s u =\lambda f(u)$ in $\Omega$; $u=0$ in $\R^N \setminus \Omega$ admits positive solution, provided $\lambda >0$ is small. They used regularity of weak solution up to the boundary of $\Omega$ and Hopf's Lemma for $(-\Delta)_p^s$.  Recently, in \cite{Nirjan2023arxiv}, the author studies $(-\Delta)^s u = g(x)\left( f(u)-a \right)$ in $\mathbb{R}^N$ with $f(0)=0, f \in \C(\R^+)$ is locally Lipschitz, satisfies subcritical and weaker A-R type growth. Whereas, $g \in L^1(\R^N) \cap L^{\infty}(\R^N)$ satisfies $|x|^{N-2s}\int_{\R^N} g(y) |x-y|^{-N+2s} \,\dy \leq C(g), x\in \R^N \setminus \{0\}, C(g)>0$. The existence of a positive solution is obtained employing similar techniques as in \cite{Alves2019}.

In this paper, we aim to study \eqref{Main problem}, a non-local analogue of \cite{Santos2023}. To our knowledge, non-linear non-local semipositone problems on the whole of $\mathbb{R}^N$ have not been addressed in the literature. The weight function $g$ falls within both $L^{1}(\mathbb{R}^N)$ and $L^\infty(\mathbb{R}^N)$ spaces, adheres to the bounds specified in \eqref{g1 new1}. Meanwhile, the function $f$ meets 
subcritical and A-R type growth, as outlined in \ref{f1} and \ref{f2}. Depending on the parameter $a$, our principal goal is to establish the existence of a positive solution to \eqref{Main problem}. The nonsmooth variational technique as in \cite{Santos2023} is not readily adapted for $s \in (0,1)$ due to the regularity constraints associated with the critical points of the non-smooth energy functional of \eqref{Main problem}. This leads us to follow a different approach from \cite{Santos2023}. We consider an energy functional associated with \eqref{Main problem}, which has a $\C^1$ variational structure (see \eqref{2.1}). Applying the mountain pass theorem, we establish the existence of a mountain pass critical point for the energy functional, which corresponds to a mountain pass solution of \eqref{Main problem}. The following theorem combines our main results.
\begin{theorem}\label{Theorem 1.1}
Let $s\in (0,1), N \ge 1$, and $p \in (1, \frac{N}{s})$. Assume that $f$ satisfies \rm{\ref{f1}} and \rm{\ref{f2}}. Let $g$ be a positive function with $g \in L^1(\R^N) \cap L^{\infty}(\R^N)$. Then the following holds:
 \begin{enumerate}
     \item[(a)] There exists $a_1>0$ such that for each $a \in (0,a_1)$, \eqref{Main problem} admits a mountain pass solution $u_a$. Moreover, there exists a constant $C>0$ such that $\|u_a\|_{s,p} \leq C$ for all $a \in (0,a_1)$.
     \item[(b)] In addition, we assume that $f$ satisfies the following condition at infinity:
    \begin{enumerate}[label={($\bf \tilde{f{\arabic*}}$)}]
    \setcounter{enumi}{0}
    \item \begin{center}
    $\displaystyle \lim\limits_{t \rightarrow \infty} \frac{f(t)}{t^{\frac{Np}{N-sp}-1}} = 0$.
    \end{center}
   \end{enumerate}
    Then for every $r \in [\frac{Np}{N-sp}, \infty]$, there exists $C=C(r, N, s, p, f, g, a_1)$ such that
    \begin{align}
\norm{u_a}_{L^r(\R^N)} \le C, \quad \forall \, a \in (0, a_1).
    \end{align}
    \item[(c)] Further, let $p >\frac{2N}{N+2s}$. Then there exists $\tilde{a} \in (0,a_1)$ such that $u_a \ge 0$  a.e. in $\R^N$ for every $a \in (0, \tilde{a})$.
    \item[(d)] Furthermore, suppose $g$ satisfies the following bound:
\begin{equation}\label{g1 new1}
      g(x) \leq \frac{B}{|x|^{\beta(p-1)+sp}}, \text{ for some constant } B>0 \text{ and } x\neq 0,
  \end{equation}
  where $\frac{N-sp}{p-1}<\beta<\frac{N}{p-1}$. Then, there exists $\hat{a} \in (0, \tilde{a})$ such that $u_a >0$ a.e. in $\R^N$ for every $a \in (0, \hat{a})$.
 \end{enumerate}
 \end{theorem}
To show the positivity of the solution $u_a$, we first show that the sequence of solutions $\{ u_a \}$ uniformly converges to a positive function in $\C(\R^N)$ as $a \ra 0$ (see Proposition \ref{properties of limit function}, Proposition \ref{unifrom convergence}, and Proposition \ref{positivity of tilde u}). For $a$ near zero, we obtain an explicit positive subsolution of \eqref{Main problem} on the exterior of a ball in $\R^N$, following the approach in \cite[Lemma 3.4]{Brasco2016_Optimaldeacy}. Subsequently, we obtain the positivity of $u_a$ using the comparison principle \cite[Theorem 2.7]{Brasco2016_Optimaldeacy}. The strategy to prove Theorem \ref{Theorem 1.1}-((c) and (d)) is different from \cite{Nirjan2023arxiv, Alves2019}, where the Riesz potential for a linear operator plays a major role.

The rest of the paper is organized as follows. In Section \ref{functional framework}, we set up a functional framework for \eqref{Main problem}. Section \ref{Existence section} covers the existence and various qualitative properties of solutions to \eqref{Main problem}. This section contains the proof of Theorem \ref{Theorem 1.1}. In the Appendix, we provide some technical lemmas. 

\section{Functional frameworks for the problem}\label{functional framework}

To obtain the existence of non-trivial solutions to \eqref{Main problem}, this section studies variational settings. We fix some notations that will be used throughout this paper for brevity. 

\noi \textbf{Notation:} 
(i) We denote $X$ as a real Banach space endowed with the norm $\|\cdot\|_X$.
\\
\noi (ii) $X^*$ denotes the dual of $X$.\\
\noi (iii) We denote $\|\cdot\|_{*}$ as the norm on $(\Dsp)^*$. \\
\noi (iv) For $p \in [1, \infty]$, the $L^p(\R^N)$ norm of a function $u$ is denoted as $\norm{u}_p$. \\
\noi (v) For $s \in (0,1)$ and $p \in (1, \frac{N}{s})$, $p_s^{*}=\frac{Np}{N-sp}$ is the non-local critical exponent. \\
\noi (vi) We denote $\Phi : \R \rightarrow \R$ which is defined as $\Phi(t)=|t|^{p-2}t$. \\
\noi (vii) We denote $C, \widetilde{C}, C_1, C_2, C_3$ as positive constants. \\
\noi (viii) $B_r$ denotes an open ball of radius $r$ with centre at origin. \\
\noi (ix) For $A \subset \R^N$, $A^c$ denotes the complement of $A$, i.e., $A^c=\R^N \setminus A$. \\
\noi (x) For $s \in (0,1)$ and $p \in (1,\infty)$, 
$$L_{sp}^{p-1}(\R^N):= \left\{ u \in L_{\text{loc}}^{p-1}(\R^N) : \int_{\R^N} \frac{\abs{u(x)}^{p-1}}{(1+\abs{x})^{N+sp}}\,\dx < \infty \right\}.$$ 
\\
\noi (xi) For $s \in (0,1)$ and $p \in (1,\infty)$, 
$$ W^{s,p}(\Omega):= \left\{ u \in L^p(\Omega) : \abs{u}^{p}_{W^{s,p}(\Omega)}:=\iint_{\Omega \times \Omega} \frac{\abs{u(x)-u(y)}^p}{|x-y|^{N+sp}}\,\dxy <\infty  \right\},$$
and 
$$W^{s,p}_{\text{loc}}(\Omega) := \left\{ u \in L^p_{\text{loc}}(\Omega): u \in W^{s,p}(\Omega_1), \text{  for any relatively compact open set }  \Omega_1 \subset \Omega \right\}.$$

For $g \in L^1(\R^N) \cap L^\infty(\R^N)$ and $a \geq 0$, we define the maps $K_a$ and $I_a$ on $\Dsp$ as follows 
 \[K_a(u) := \int_{\R^N} g(x)F_a(u(x))\,\dx, \text{ and } I_a(u) := \frac{1}{p} \|u\|_{s,p}^p -K_a(u).\]
 Notice that $K_a, I_a \in \C^1(\Dsp,\R)$ and the corresponding Fréchet derivatives are given by
 \begin{equation}\label{2.1}
     \begin{split}
    & K_a'(u)(v) = \int_{\R^N} g(x)f_a(u) v\,\dx, \text{ and } \\
    & I_a'(u)(v) = \iint_{\R^{2N}} \frac{\Phi(u(x)-u(y) )(v(x)-v(y))}{|x-y|^{N+sp}}\, \dxy -K_a'(u)(v), \; \forall \, v \in \Dsp. 
     \end{split}
 \end{equation}
Moreover, every critical point of $I_a$ corresponds to a solution of \eqref{Main problem}. Before discussing further properties of $K_a, I_a$, we identify the upper and lower bounds of $f, f_a,$ and their primitives. Clearly, the function $f_a \in \C(\R)$ and its primitive $F_a$ is defined as $F_a(t) = F(t) - at$ for $t \ge 0$, $F_a(t) = -\frac{at^2}{2}-at$ for $t \in [-1,0]$ and $F_a(t) =\frac{a}{2}$ for $t \leq -1$. From \ref{f1},
\begin{align*}
  &\lim\limits_{t \rightarrow 0^+} \frac{f(t)}{t^{p-1}}= 0 \Rightarrow \text{ for every } \epsilon >0, \text{ there exists }  t_1(\epsilon)>0 \text{ such that }  f(t)< \epsilon t^{p-1} \text{ for } t \in (0, t_1(\epsilon)). \\
  & \lim\limits_{t \rightarrow \infty} \frac{f(t)}{t^{\gamma-1}} \leq C(f) \Rightarrow f(t) \leq C(f,t_1(\epsilon)) t^{\gamma-1} \text{ for } t \geq t_1(\epsilon).
\end{align*}
Hence, we have the following bounds for $\gamma \in (p, p^*_s]$: 
\begin{align}\label{2.2}
     |f_a(t)| \leq \epsilon |t|^{p-1} + C(f,t_1(\epsilon)) |t|^{\gamma-1} +a \text{ and } |F_a(t)| \leq \epsilon |t|^{p} + C(f,t_1(\epsilon)) |t|^{\gamma}+a|t| \text{ for } t \in \R. 
\end{align} 
Again using the subcritical growth on $f$, $f(t) \leq C(f)t^{\gamma-1}$, for $t>t_2$, for some $t_2>0$. The continuity of $f$ infers that $f(t) \leq C$ on $[0,t_2]$. Hence for $a \in (0,\Tilde{a})$, we obtain
\begin{align}\label{2.3}
    |f_a(t)| \leq C(1+ |t|^{\gamma-1}) \text{ and } |F_a(t)| \leq C(|t|+|t|^\gamma) \text{ for } t \in \R, \text{ where } C=C(f,t_2,\Tilde{a}). 
\end{align}
By \ref{f2} and the continuity of $F(t)$, there exist $M_1,M_2>0$ such that 
\begin{align} \label{2.4}
    F(t) \geq M_1 t^\vartheta - M_2, \; \forall \, t \geq 0.  
\end{align}

\begin{remark}[A-R condition of $f_a$]
    For $t>t_0$, it follows from \rm{\ref{f2}} that \[\vartheta F_a(t) = \vartheta F(t) - \vartheta at \leq tf(t)-at = tf_a(t).\] 
For $t \in [0,t_0]$, by continuity of $F$, there exists $M >0$ independent of $a$, such that \[\vartheta F_a(t)= \vartheta F(t)-\vartheta at \leq M - at \leq M-at + tf(t)= tf_a(t)+M.\] 
For $t \in [-1,0]$ and $a \in (0, \tilde{a})$, observe that $\vartheta F_a(t) \leq -\vartheta at -\frac{a t^2}{2}$ and $tf_a(t) = -at^2 -at$. Then we have the following estimate:
\begin{align*}
    \vartheta F_a(t) - tf_a(t) \leq -(\vartheta -1) a t + \frac{at^2}{2} \leq -(\vartheta -1) \tilde{a} t + \frac{\tilde{a}}{2} \leq \vartheta \tilde{a}.
\end{align*}
For $t \leq -1$, we have
\[\vartheta F_a(t) = \frac{\vartheta a}{2} \leq \frac{\vartheta\tilde{a}}{2}  \leq tf_a(t) + \vartheta \tilde{a}. \]
By choosing $M_3 = \max\{M,\vartheta \tilde{a}\}$, we obtain the following Ambrosetti-Rabinowitz (A-R) condition for $f_a$:
\begin{align} \label{2.5}
    \vartheta F_a(t) \leq t f_a(t)+M_3,\;  \forall \, t \in \R \text{ and } a \in (0,\tilde{a}),
\end{align} 
where $M_3$ is independent of $a$ and $t$.
\end{remark}

The following proposition states some compact embeddings of the solution space for \eqref{Main problem} into the spaces of locally integrable functions and weighted Lebesgue spaces.

\begin{proposition}\label{compactness weight g}
   Let $p \in (1, \frac{N}{s})$ and $q \in [1,p_s^{*})$. Then the following hold:
   \begin{enumerate}
       \item[\rm{(i)}] The embedding $\Dsp \hookrightarrow L_{\text{loc}}^q(\R^N)$ is compact.
       \item[\rm{(ii)}] Suppose $g \in L^{\alpha}(\R^N)$ for $\alpha= \frac{p_s^{*}}{p_s^{*}-q}$. Then the embedding $\Dsp \hookrightarrow L^{q}(\R^N, |g|)$ is compact.
   \end{enumerate}
\end{proposition}

\begin{proof}
   \noi (i) Proof follows using \cite[Lemma A.1]{Brasco2014_fractional_Cheeger} and the same arguments as given in \cite[Proposition 2.1]{Nirjan2023arxiv}.

   \noi (ii) Assume that $u_n \rightharpoonup u$ in $\Dsp$. We need to prove $u_n \rightarrow u$ in $L^q(\R^N,|g|)$. The space $\C_c(\R^N)$ is dense in $L^{\alpha}(\R^N)$ and hence for every $\epsilon>0$ we take $g_\epsilon \in \C_c(\R^N)$ with $K:= \mathrm{supp}(g_\epsilon)$ such that \[\|g-g_\epsilon\|_{\alpha}<\frac{\epsilon}{2L},\] where $L:= \sup\{\|u_n-u\|_{p_s^{*}}^q : n \in \N\}$. Now, using the triangle and H\"{o}lder inequalities, we obtain the following estimates:
\begin{align}\label{2.6}
     \int_{\R^N} |g||u_n-u|^q\,\dx &\leq \int_{\R^N} |g-g_\epsilon||u_n-u|^q\,\dx + \int_{\R^N} |g_\epsilon||u_n-u|^q\,\dx \notag\\
    & \leq \|g-g_\epsilon\|_{\alpha} \|u_n-u\|_{p_s^{*}}^q + \int_{K}|g_\epsilon||u_n-u|^q\,\dx \leq \frac{\epsilon}{2} + M \int_{K} |u_n-u|^q\,\dx.
\end{align}
The above constant $M$ is the upper bound of $g_\epsilon$ on the compact set $K$. Moreover, the embedding $\Dsp \hookrightarrow L^{q}_{\text{loc}}(\R^N)$ is compact. Therefore, there exists $n_1 \in \N$ such that up to a subsequence
\begin{equation}\label{2.7}
    \int_{K} |u_n-u|^q\,\dx < \frac{\epsilon}{2M}, \;  \forall \, n \geq n_1.
\end{equation}
From \eqref{2.6} and \eqref{2.7}, we deduce that
\[\int_{\R^N} |g||u_n-u|^q\,\dx < \epsilon, \;  \forall \, n \geq n_1. \]
Since $\epsilon>0$ is arbitrary, we get the desired result.
\end{proof}

Now, we prove the compactness of $K_a$ using a similar splitting argument as above.

\begin{proposition} \label{locally lipschitz K_a}
Let $p \in (1, \frac{N}{s})$. Assume that $f$ satisfies \rm{\ref{f1}} and $g \in L^1(\R^N)\cap L^\infty(\R^N)$. Then $K_a$ is compact on $\Dsp$ for every $a \geq 0$.
\end{proposition}
\begin{proof}
 Let $u_n \rightharpoonup u$ in $\Dsp$. Since $\C_c(\R^N)$ is dense in both $L^{\frac{p_s^*}{p_s^*-1}}(\R^N)$ and $L^{\frac{p_s^*}{p_s^*-\gamma}}(\R^N)$, for every given $\epsilon>0$ we take $g_\epsilon \in \C_c(\R^N)$ such that
 \begin{align}\label{2.8}
     \|g-g_\epsilon\|_{\frac{p_s^*}{p_s^*-1}}+\|g-g_\epsilon\|_{\frac{p_s^*}{p_s^*-\gamma}}< \frac{\epsilon}{L},
 \end{align}
where $L:= \sup\{\|u_n\|_{p_s^*} + \|u\|_{p_s^*} +\|u_n\|_{p_s^*}^\gamma + \|u\|_{p_s^*}^\gamma:n \in \N\}$. Since the sequence $\{\|u_n\|_{s,p}\}$ is bounded in $\mathbb{R}^+$, using the continuous embedding $\Dsp \hookrightarrow L^{p_s^*}(\R^N)$ we see that $\{\|u_n\|_{p_s^*}\}$ is also bounded in $\mathbb{R}^+$. Therefore, $L<\infty$.
For $a\geq 0$, we write 
    \begin{align} \label{2.9}
        &\big|K_a(u_n)-K_a(u)\big| \le \int_{\R^N} g(x) \left| \big(F_a(u_n(x))-F_a(u(x)) \big)\right| \,\dx  \notag\\
        &\leq \int_{\R^N} |g-g_\epsilon|\big|F_a(u_n(x))-F_a(u(x)) \big|\,\dx + \int_{\R^N} |g_\epsilon|\big|F_a(u_n(x))-F_a(u(x)) \big|\,\dx := \text{I} + \text{II}.
    \end{align}
 Now using \eqref{2.3} and H\"{o}lder's inequality with \eqref{2.8}, we estimate the first integral as
 \begin{align}\label{2.10}
    \text{I}&\leq \int_{\R^N} |g-g_\epsilon|\big(|F_a(u_n)|+ |F_a(u)| \big)\,\dx \leq C \int_{\R^N} |g-g_\epsilon|\big(|u_n| + |u_n|^\gamma + |u| + |u|^\gamma \big)\,\dx \notag\\
    & \leq C \big( \|g-g_\epsilon\|_{\frac{p_s^*}{p_s^*-1}} (\|u_n\|_{p_s^*} + \|u\|_{p_s^*}) + \|g-g_\epsilon\|_{\frac{p_s^*}{p_s^*-\gamma}} (\|u_n\|_{p_s^*}^\gamma + \|u\|_{p_s^*}^\gamma)\big) < C \epsilon,
 \end{align}
where $C=C(f,a)$.  Further, we estimate the following integral
 \begin{align} \label{2.11}
    \text{II}= \int_{\R^N} |g_\epsilon|\big|F_a(u_n)-F_a(u) \big|\,\dx \leq M \int_{K}\big|F_a(u_n)-F_a(u) \big|\,\dx,
 \end{align}
where $K$ is the support of $g_\epsilon$ and $M=\|g_\epsilon\|_{\infty}$.  Since $\Dsp$ is compactly embedded into $L^{\gamma}_{\text{loc}}(\R^N)$ (Proposition \ref{compactness weight g}), up to a subsequence we get $u_n \rightarrow u$ in $L^{\gamma}(K)$, and subsequently $u_n(x) \rightarrow u(x)$ a.e. in $K$. Moreover, $|F_a(u_n)| \leq C (|u_n| + |u_n|^\gamma)$ and $\int_{K}|u_n|\,\dx \rightarrow \int_{K}|u|\,\dx,~~\int_{K}|u_n|^\gamma\,\dx \rightarrow \int_{K}|u|^\gamma\,\dx$. Therefore, the generalized dominated convergence theorem yields $F_a(u_n) \rightarrow F_a(u)$ in $L^{1}(K)$. Thus from \eqref{2.11},
 \[\int_{\R^N} |g_\epsilon|\big|F_a(u_n)-F_a(u) \big|\,\dx \rightarrow 0  ~~\text{ as } n \rightarrow \infty. \]
Now we conclude from \eqref{2.9} and \eqref{2.10} that $K_a(u_n) \rightarrow K_a(u)$ as $n \rightarrow \infty$.
\end{proof}

Now, depending on the values of $a$, we verify the mountain pass geometry for the energy functional $I_a$. 

\begin{lemma}\label{MP1}
Let $p \in (1, \frac{N}{s})$. Let $f$ satisfies \rm{\ref{f1}}, \rm{\ref{f2}} and $g \in L^1(\R^N)\cap L^\infty(\R^N)$ be positive. Then the following hold:
\begin{itemize}
       \item[\rm{(i)}] There exists $\beta, \delta>0$, and $a_1>0$ such that $I_a(u) \geq \delta$ for $\norm{u}_{s,p}=\beta$ whenever $a \in (0,a_1)$.
       \item[\rm{(ii)}] There exists $v \in \Dsp$ with $\|v\|_{s,p}> \beta$ such that $I_a(v)<0$, for all $a > 0$. 
   \end{itemize} 
\end{lemma}
\begin{proof}
  (i)  The functional $I_a : \Dsp \rightarrow \R$  is given by
    \[I_a(u)= \frac{1}{p}\Gagp - \int_{\R^N} g(x)F_a(u)\, \dx.\]
    Using \eqref{2.2}, we estimate
\begin{align*}
    \int_{\R^N} g(x)F_a(u)\,\dx &\leq \int_{\R^N} g(x)\big(\epsilon |u|^p+ C(f,\epsilon) |u|^\gamma+ a|u| \big)\, \dx \\
    &= \epsilon \int_{\R^N} g(x)|u|^p\,\dx +C(f,\epsilon) \int_{\R^N} g(x)|u|^\gamma\,\dx + a \int_{\R^N} g(x)|u|\,\dx\\
    &\leq \epsilon C_1 \norm{u}_{s,p}^p + C(f,\epsilon) C_2 \norm{u}_{s,p}^\gamma + a C_3 \|u\|_{s,p},
\end{align*}  
where $C_1, C_2$ and $C_3$ are the embedding constants of $\Dsp \hookrightarrow L^q( \R^N,g); q\in [1, p^*_s)$ (Proposition \ref{compactness weight g}). In particular, for $\norm{u}_{s,p}=\beta$,
\begin{align}\label{2.12}
    I_a(u) &\geq \frac{1}{p} \beta^p - \epsilon C_1 \beta^p - C(f,\epsilon) C_2 \beta^\gamma - a C_3 \beta \no \\
    &= \beta^p \bigg(\frac{1}{p} - \epsilon C_1 - C(f,\epsilon) C_2 \beta^{\gamma-p} \bigg)-a C_3 \beta.
\end{align}
We choose $\epsilon < (pC_1)^{-1}$. Then we write $I_a(u) \geq A(\beta) - aC_3 \beta$, where $A(\beta) = C \beta^p (1-\widetilde{C}\beta^{\gamma -p})$ with $C,\widetilde{C}$ independent of $a$. Let $\beta_1$ be the first non-trivial zero of $A$. For $\beta < \beta_1$, we fix $a_1 \in (0, \frac{A(\beta)}{C_3 \beta})$ and $\delta = A(\beta)-a_1 C_3 \beta$. Thus, by \eqref{2.12} we get $I_a(u) \geq \delta$ for every $a \in (0,a_1)$. \\

\noi (ii) Let $\varphi \in C_c^\infty(\R^N) \setminus \{0\}, \varphi \geq 0$ and $\|\varphi\|_{s,p}=1$. For $t \geq 0$, we have
    \[ I_a(t \varphi) = \frac{t^p}{p} \iint_{\R^{2N}}\frac{|\varphi(x)-\varphi(y)|^p}{|x-y|^{N+sp}}\, \dxy
    - \int_{\R^N}g(x) \left(F(t \varphi) -at\phi  \right) \,\dx.\]
Now using the A-R condition \eqref{2.4} of $F$, we obtain the following
\begin{align}\label{2.13}
    I_a(t \varphi) &\leq \frac{t^p}{p} \|\varphi\|_{s,p}^p
    -M_1t^\vartheta \int_{\R^N}g(x)(\varphi(x))^\vartheta\,\dx + M_2 \int_{\R^N}g(x)\,\dx+ at \int_{\R^N}g(x) \varphi(x)\,\dx  \notag\\
    &\leq  \frac{t^p}{p} -M_1 t^\vartheta \int_{\R^N}g(x)(\varphi(x))^\vartheta\,\dx+ M_2 \|g\|_1  + at \norm{\varphi}_{\infty} \norm{g}_1.
\end{align}
Using the fact $\vartheta >p >1$, it is easy to see that $I_a(t \varphi) \rightarrow - \infty$ as $t \rightarrow +\infty$. Thus, there exists a $t_1 > \beta$ such that $I_a(t\varphi)<0$ for $t>t_1$. Therefore, the required function is $v= t\varphi$ with $t>t_1$.
\end{proof}

\begin{definition}[Palais Smale condition]
Let $J:\Dsp \rightarrow \R$ be a continuously differentiable functional. Then $J$ satisfies the Palais Smale (PS) condition if every sequence $\{u_n\} \subset \Dsp$ with $\{J(u_n)\}$ is bounded in $\R$ and $J'(u_n) \rightarrow 0$ in $(\Dsp)^*$ possesses a convergent subsequence.
\end{definition}

\begin{proposition} \label{Palais smale I_a}
Let $p \in (1, \frac{N}{s})$. Let $f$ satisfies \rm{\ref{f1}}, \rm{\ref{f2}}, and  $g \in L^1(\R^N)\cap L^\infty(\R^N)$ be positive. Then $I_a$ satisfies the (PS) condition for every $a \geq 0$.
\end{proposition}
\begin{proof}
Let $\{u_n\}$ be a sequence in $\Dsp$ such that $\{I_a(u_n)\}$ is bounded in $\R$ and $I_a'(u_n) \rightarrow 0$ in $(\Dsp)^*$. We need to show that the sequence $\{u_n\}$ has a strongly convergent subsequence in $\Dsp$. First, we show that $\{u_n\}$ is a bounded sequence in $\Dsp$. Since $\{I_a(u_n)\}$ is bounded, we have $|I_a(u_n)|\leq M$ for some $M>0$, which further implies
\begin{equation}\label{2.14}
    \frac{1}{p}\norm{u_n}_{s,p}^p - \int_{\R^N} g(x)F_a(u_n)\,\dx \leq M, \; \forall \, n\in \N.
\end{equation}
Next, using $I_a'(u_n) \rightarrow 0$ in $(\Dsp)^*$, there exists $n_1 \in \N$ such that for every $n \geq n_1$, we have $\abs{I_a'(u_n)(u_n)} \leq \norm{u_n}_{s,p}$. Thus we obtain
\begin{equation}\label{2.15}
    -\norm{u_n}_{s,p} -\norm{u_n}_{s,p}^p \leq -\int_{\R^N} g(x)f_a(u_n)u_n\,\dx, \; \forall \, n \geq n_1.
\end{equation}
From \eqref{2.5} and \eqref{2.14}, we see that
\begin{equation}\label{2.16}
    \frac{1}{p}\norm{u_n}_{s,p}^p - \frac{1}{\vartheta}\int_{\R^N} g(x)f_a(u_n)u_n\,\dx- \frac{1}{\vartheta} M_3\norm{g}_1\leq M,
\end{equation}
and further by \eqref{2.15} and \eqref{2.16},
\begin{align*}
    \bigg(\frac{1}{p}-\frac{1}{\vartheta} \bigg)\norm{u_n}_{s,p}^p -\frac{1}{\vartheta}\norm{u_n}_{s,p} \leq M + \frac{1}{\vartheta}M_3 \norm{g}_1, \; \forall \, n \geq n_1.
\end{align*}
The above inequality infers that $\{u_n\}$ is bounded in $\Dsp$. By reflexivity of $\Dsp$, there exists $u \in \Dsp$ such that up to a subsequence $u_n \rightharpoonup u$ in $\Dsp$. Now we consider the functional $J(v) = \frac{1}{p}\|v\|_{s,p}^p$ for $v \in \Dsp$. Clearly, $J \in \C^1(\Dsp,\R)$. From \eqref{2.1}, we write $J'(u_n)(u_n -u) = I_a'(u_n)(u_n-u) + K_a'(u_n)(u_n-u)$. Now we claim that $J'(u_n)(u_n -u) \rightarrow 0$ as $n \rightarrow \infty$. First, we prove $I_a'(u_n)(u_n-u) \rightarrow 0$ as $n \rightarrow \infty$. Recall that $I_a'(u_n) \rightarrow 0 \text{ in } (\Dsp)^*$ and $\{u_n\}$ is bounded in $\Dsp$. Consequently, we deduce $| I_a'(u_n) (u_n-u) | \leq \|I_a'(u_n)\|_{*} \|u_n-u\|_{s,p} \rightarrow 0$ as $n \rightarrow \infty$. Next, we prove $K_a'(u_n)(u_n-u) \rightarrow 0$. Observe from \eqref{2.3} that
 \begin{align} \label{2.17}
     |K_a'(u_n)(u_n-u)| & \leq \int_{\R^N} g(x)|f_a(u_n)||u_n-u|\, \dx \no \\
     & \leq C(f,a) \int_{\R^N} g(x) \left( 1 + |u_n|^{\gamma-1} \right) |u_n-u|\, \dx.
 \end{align} 
The sequence $\{|u_n|^{\gamma -1}\}$ is bounded in $L^{\frac{p_s^*}{\gamma-1}}(\R^N)$ (as $\{u_n\}$ is bounded in $\Dsp$). Further, 
 \begin{align*}
     \gamma < p_s^* \Longleftrightarrow \frac{p_s^*}{p_s^*-(\gamma-1)} < p_s^*.
 \end{align*}
 Therefore, applying Proposition \ref{compactness weight g} and  $g \in L^1(\R^N) \cap L^{\infty}(\R^N)$, we get $u_n \ra u$ in $L^{\frac{p_s^*}{p_s^*-(\gamma-1)}}(\R^N,g)$. Thus, using H\"{o}lder's inequality with conjugate pair $(\frac{p_s^{*}}{\gamma-1}, \frac{p_s^*}{p_s^*-(\gamma-1)})$,
 \begin{align}
     \int_{\R^N} g(x)|u_n-u||u_n|^{\gamma-1}\,\dx \leq \norm{g}_{\infty}^{\frac{\gamma-1}{p^*_s}} \norm{u_n-u}_{L^{\frac{p_s^*}{p_s^*-(\gamma-1)}}(\R^N,g)} \||u_n|^{\gamma-1}\|_{\frac{p_s^{*}}{\gamma-1}} \rightarrow 0 \text{ as } n \rightarrow \infty.
 \end{align} 
Also, $\int_{\R^N} g(x)|u_n-u|\,\dx \rightarrow 0$ as $n \ra \infty$ (by Proposition \ref{compactness weight g}). Thus we have 
 \begin{align}\label{2.19}
     \int_{\R^N} g(x)\left( 1 + |u_n|^{\gamma-1}   \right)|u_n-u| \,\dx \rightarrow 0 \text{ as } n \rightarrow \infty.
 \end{align}
We infer from \eqref{2.17} and \eqref{2.19} that $K_a'(u_n)(u_n-u) \rightarrow 0$ as $n \rightarrow \infty$. Thus, our claim holds true. Since $J \in \C^1(\Dsp, \R)$ and $u_n \rightharpoonup u$ in $\Dsp$, we also get $J'(u)(u_n -u) \rightarrow 0$ as $n \rightarrow \infty$. Further, we estimate the following using H\"{o}lder inequality:
 \begin{align}\label{2.20}
     &J'(u_n)(u_n -u)-J'(u)(u_n -u)\no\\
     &~~=\iint_{\R^{2N}}\frac{\Phi\big(u_n(x)-u_n(y)\big)-\Phi \big(u(x)-u(y)\big)}{|x-y|^{N+sp}}\big((u_n-u)(x)-(u_n-u)(y)\big)\, \dxy \notag\\
     &~~= \Gagnp + \Gagp \notag\\
     &~~\; - \iint_{\R^{2N}}\frac{\Phi \big(u_n(x)-u_n(y)\big)\big(u(x)-u(y)\big)}{|x-y|^{N+sp}}\, \dxy- \iint_{\R^{2N}}\frac{\Phi\big(u(x)-u(y)\big)\big(u_n(x)-u_n(y)\big)}{|x-y|^{N+sp}}\, \dxy \notag\\
     &~~\geq \norm{u_n}_{s,p}^p + \norm{u}_{s,p}^p - \norm{u_n}_{s,p}^{p-1}\norm{u}_{s,p} - \norm{u}_{s,p}^{p-1}\norm{u_n}_{s,p} \notag\\
     &~~= \big(\norm{u_n}_{s,p}^{p-1}-\norm{u}_{s,p}^{p-1} \big) \big(\norm{u_n}_{s,p}-\norm{u}_{s,p} \big) \geq 0.
 \end{align}
Therefore, taking limit as $n \rightarrow \infty$ in \eqref{2.20}, we get
 $\norm{u_n}_{s,p} \rightarrow \norm{u}_{s,p}$. Since $\Dsp$ is uniformly convex Banach space, 
 $u_n \rightarrow u$ in $\Dsp$. 
\end{proof}

\section{Existence and qualitative properties of the mountain pass solutions} \label{Existence section}

The following theorem states the existence and uniform boundedness of the mountain pass solutions of \eqref{Main problem}.

\begin{theorem} \label{Existence and uniform boundedness}
    Let $p \in (1, \frac{N}{s})$. Assume that $f$ satisfies \rm{\ref{f1}} and \rm{\ref{f2}}. Let $g \in L^1(\R^N) \cap L^{\infty}(\R^N)$ be positive. Let $a_1 >0$ be given in Lemma \ref{MP1}.
    Then for each $a \in (0,a_1)$, \eqref{Main problem} admits a mountain pass solution $u_a$. Moreover, there exists $C>0$ such that $\|u_a\|_{s,p} \leq C$ for all $a \in (0,a_1)$.
\end{theorem}
\begin{proof}
  Consider $a_1, \delta,v$ as given in Lemma \ref{MP1}. For $a \in (0,a_1)$, using Lemma \ref{MP1} and Proposition \ref{Palais smale I_a}, we observe that all the hypotheses of the mountain pass theorem \cite[Theorem 2.1]{Ambrosetti1973} are verified. Therefore, by \cite[Theorem 2.1]{Ambrosetti1973}, there exists a non-trivial critical point $u_a \in \Dsp$ of $I_a$ satisfying
   \begin{equation} \label{3.1}
       I_a(u_a) = c_a = \inf\limits_{\gamma \in \Gamma_v}\max\limits_{t \in [0,1]}I_a(\gamma(t)) \geq \delta \text{  and  } I_a'(u_a)=0,
   \end{equation}
where $\Gamma_v := \{ \gamma \in C([0,1], \Dsp):\gamma(0)=0  \text{ and } \gamma(1)=v \}$ and $c_a$ is the mountian pass level associated with $I_a$. Thus, $u_a$ is a non-trivial solution to \eqref{Main problem}. To prove the uniform boundedness of $u_a$ in $\Dsp$, we first show that the set $\{I_a(u_a): a \in (0,a_1)\}$ is uniformly bounded. We define a path $\Tilde{\gamma}:[0,1] \rightarrow \Dsp$ by $\Tilde{\gamma}(\sigma)=\sigma v$, where $v=t\varphi$ for some $t>t_1$ (for $t_1$ as in Lemma \ref{MP1}-(ii)), $\varphi \in \C_c^{\infty}(\R^N)\setminus\{0\},~ \varphi \geq 0$ and $\norm{\varphi}_{s,p}=1$. We see that $\Tilde{\gamma} \in \Gamma_v$ because $\Tilde{\gamma}(0)=0$ and $\Tilde{\gamma}(1)=v$. Now from \eqref{3.1}, we have
\begin{equation} \label{3.2}
       I_a(u_a) = c_a = \inf\limits_{\gamma \in \Gamma_v}\max\limits_{t \in [0,1]}I_a(\gamma(t)) \leq \max\limits_{\sigma \in [0,1]}I_a(\Tilde{\gamma}(\sigma)) = \max\limits_{\sigma \in [0,1]}I_a(\sigma t\varphi).
   \end{equation}
Now estimating as in \eqref{2.13}, we obtain
 \begin{align}\label{3.3}
    I_a(\sigma t \varphi) &\leq \frac{{\sigma}^pt^p}{p} \norm{\varphi}_{s,p}^p
    -M_1 {\sigma}^\vartheta t^\vartheta \int_{\R^N}g(x)(\varphi(x))^\vartheta\,\dx + M_2 \int_{\R^N}g(x)\,\dx+ a {\sigma} t \int_{\R^N}g(x) \varphi(x)\,\dx  \notag\\
    &\leq  \frac{t^p}{p} + M_2 \|g\|_1  + a_1 t C_1 \norm{\varphi}_{s,p}.
\end{align}
From \eqref{3.2} and \eqref{3.3}, there exists  $C= C(N,s,p,M_2,g,a_1)$ such that
 \begin{equation}\label{3.4}
     I_a(u_a) \leq C, \text{ for all } a \in (0,a_1).
 \end{equation}
 Using the uniform boundedness of $I_a(u_a)$, we will show that the solutions $\{u_a\}$ are uniformly bounded in $\Dsp$. From \eqref{3.1}, we write
 \begin{equation}\label{3.5}
     \norm{u_a}_{s,p}^p - \int_{\R^N}g(x)f_a(u_a)u_a\,\dx = 0.
 \end{equation}
 Also, by \eqref{3.4} we have
 \begin{equation}\label{3.6}
     \frac{1}{p}\norm{u_a}_{s,p}^p - \int_{\R^N}g(x)F_a(u_a)\,\dx \leq C.
 \end{equation}
Now first multiplying \eqref{3.5} by $\frac{1}{\vartheta}$ and then subtracting into \eqref{3.6} gives the following
\[ \bigg(\frac{1}{p}-\frac{1}{\vartheta}\bigg)\norm{u_a}_{s,p}^p + \int_{\R^N}g(x) \bigg(\frac{1}{\vartheta}f_a(u_a)u_a- F_a(u_a) \bigg)\,\dx \leq C,\]
and combining the above with \eqref{2.5} yields
\[ \bigg(\frac{1}{p}-\frac{1}{\vartheta}\bigg)\norm{u_a}_{s,p}^p -\frac{1}{\vartheta}M_3\|g\|_1 \leq C.\]
Thus, there exists $C=C(N,s,p,f,g,a_1)$ such that $\norm{u_a}_{s,p} \leq C$ for every $a \in (0,a_1)$.
\end{proof}

From Theorem \ref{Existence and uniform boundedness} and the embedding $\Dsp \hookrightarrow L^{p^*_s}(\R^N)$ it is clear that $\{ u_a: a \in (0, a_1)\}$ is uniformly bounded on $L^{p^*_s}(\R^N)$.  In the following proposition, using the Moser iteration technique, we discuss the uniform boundedness of $\{ u_a: a \in (0, a_1) \}$ over $L^r(\R^N)$ for every $r \in [p^*_s, \infty]$.

\begin{remark}\label{pm part}
    For $h \in \Dsp$, we check that $h^{\pm} \in \Dsp$ and $\norm{h^{\pm}}_{s,p} \le \norm{h}_{s,p}$. We verify this for the positive part of $h$; a similar argument holds for the negative part. Set $A_h := \{ x \in \R^N : h(x) \ge 0\}$. Then we see that
    \begin{align*}
      \norm{h^+}_{s,p}^{p} = & \iint_{\R^{2N}} \frac{ \abs{h^+(x) - h^+(y)}^{p} }{\abs{x-y}^{N+sp}} \, \dxy = \left( \iint_{A_h \times A_h} + 2 \iint_{A_h \times (A_h)^c} \right) \frac{ \abs{h^+(x) - h^+(y)}^{p} }{\abs{x-y}^{N+sp}} \, \dxy \\
        & = \iint_{A_h \times A_h} \frac{ \abs{h(x) - h(y)}^{p} }{\abs{x-y}^{N+sp}} \, \dxy + 2 \iint_{A_h \times (A_h)^c} \frac{\abs{h(x)}^{p} }{\abs{x-y}^{N+sp}} \, \dxy \\
        & \le \left( \iint_{A_h \times A_h} + 2 \iint_{A_h \times (A_h)^c} + \iint_{(A_h)^c \times (A_h)^c}  \right)  \frac{ \abs{h(x) - h(y)}^{p} }{\abs{x-y}^{N+sp}} \, \dxy = \norm{h}_{s,p}^{p}.
    \end{align*}
\end{remark}

\begin{proposition}\label{regularity}
   Let $p \in (1, \frac{N}{s})$. Let $f$ satisfies \rm{\ref{f1}}, \rm{\ref{f2}}, and  $g \in L^1(\R^N)\cap L^\infty(\R^N)$ be positive. In addition, we assume that $f$ satisfies the following condition at infinity:
    \begin{enumerate}[label={($\bf \tilde{f{\arabic*}}$)}]
    \setcounter{enumi}{0}
    \item \label{f1 new} \begin{center}
        $\displaystyle \lim\limits_{t \rightarrow \infty} \frac{f(t)}{t^{p^*_s-1}} = 0$.  
    \end{center}
   \end{enumerate}
    Then for every $r \in [p^*_s, \infty]$ there exists $C=C(r,N,s,p,f,g,a_1)>0$ such that
    \begin{align}\label{3.7 unifrom bound}
        \norm{u_a}_{r} \le C, \quad \forall \, a \in (0, a_1).
    \end{align}
\end{proposition}
\begin{proof}
  \textbf{Uniform bound of the positive part:} For $M > 0$, define $(u_a^+)_M= \min\{(u_a)^+, M\}$ for every $a \in (0, a_1)$. Clearly $(u_a^+)_M \ge 0$ and $(u_a^+)_M \in L^{\infty}(\R^N) \cap \Dsp$. Fixed $\sigma\geq1$, define $\phi=((u_a^+)_M)^{\sigma}$. Now we claim that $\phi \in \Dsp$. First, we recall the following inequality from \cite[(2.4), Page 1359]{Ims}: for any $\alpha,\beta \in \mathbb{R}$ and $\sigma \geq 1$, we have
  \begin{align}\label{inequality-squassina}
    \left| |\alpha|^{\sigma-1} \alpha - |\beta|^{\sigma-1} \beta  \right| \leq \sigma \left( |\alpha|^{\sigma-1} + |\beta|^{\sigma-1} \right)|\alpha-\beta|. \tag{A}
  \end{align}
Taking $\alpha= (u_a^+)_M(x)$ and $\beta=(u_a^+)_M(y)$ into \eqref{inequality-squassina} and using the fact that $0 \leq (u_a^+)_M \in L^\infty(\R^N)$, we get
\begin{align*}
    \left| \phi(x) - \phi(y)  \right| &\leq \sigma \left( |(u_a^+)_M(x)|^{\sigma-1} + |(u_a^+)_M(y)|^{\sigma-1} \right) \left|(u_a^+)_M(x)-(u_a^+)_M(y)\right|\\
    & \leq 2 \sigma \|(u_a^+)_M\|_{\infty}^{\sigma-1} \left|(u_a^+)_M(x)-(u_a^+)_M(y)\right|.
\end{align*}
Since $(u_a^+)_M \in \Dsp$, we have $\norm{(u_a^+)_M}_{s,p} < \infty$ and which further implies $\norm{\phi}_{s,p}< \infty$ i.e.,
\begin{align*}
    \norm{\phi}_{s,p}^{p} = & \iint_{\R^{2N}} \frac{ \abs{\phi(x) - \phi(y)}^{p} }{\abs{x-y}^{N+sp}} \,\dxy \leq \left(2 \sigma \|(u_a^+)_M\|_{\infty}^{\sigma-1} \right)^p \iint_{\R^{2N}} \frac{ \abs{(u_a^+)_M(x) - (u_a^+)_M(y)}^{p} }{\abs{x-y}^{N+sp}} \,\dxy < \infty.
\end{align*} 
Also, $(u_a^+)_M \in \Dsp \cap L^\infty(\R^N)$ implies $(u_a^+)_M \in L^{p_s^*}(\R^N) \cap L^\infty(\R^N)$, which further implies $(u_a^+)_M \in L^r(\R^N)$ for all $r \in [p_s^*,\infty]$. Hence, we get $\phi = (u_a^+)_M^\sigma \in L^{p_s^*}(\R^N)$. By definition of $\Dsp$ given in \eqref{dsp}, we conclude that $\phi \in \Dsp$.

By Theorem \ref{Existence and uniform boundedness}, $u_a$ is a weak solution of \eqref{Main problem}. Taking $\phi$ as a test function, we write 
   \begin{align*}
       \iint_{\R^{2N}} \frac{\Phi\big(u_a(x)-u_a(y)\big)\big(\phi(x) - \phi(y)\big)}{|x-y|^{N+sp}}\, \dxy =  \int_{\R^N } g(x) f_a(u_a) \phi(x) \, \dx.
   \end{align*}
   Using \ref{i2}, Remark \ref{pm part}, and $\Dsp \hookrightarrow L^{p^*_s}(\R^N)$, we have the following lower bound of the left hand side: 
   \begin{align*}
       & \iint_{\R^{2N}} \frac{\Phi\big(u_a(x)-u_a(y)\big)\big(\phi(x) - \phi(y)\big)}{|x-y|^{N+sp}}\, \dxy \\
       &~~\ge \frac{ p^p}{(\sigma + p -1)^p} 
       \iint_{\R^{2N}} \frac{\left| \big((u_a^+)_M(x)\big)^{\frac{\sigma+p-1}{p}}-\big((u_a^+)_M(y)\big)^{\frac{\sigma+p-1}{p}} \right|^p}{\abs{x-y}^{N+sp}} \, \dxy \\
       &~~~\ge \frac{C(N,s,p) p^p}{(\sigma+p-1)^p}
       \left(\int_{\R^N}\left( \big((u_a^+)_M(x)\big)^{\frac{\sigma+p-1}{p}} \right)^{p^*_{s}}\,\dx\right)^{\frac{p}{p^*_{s}}}.
   \end{align*}
   Hence we get
   \begin{align*}
       \frac{C(N,s,p) p^p}{(\sigma+p-1)^p}
       \left(\int_{\R^N}\left(\big((u_a^+)_M(x)\big)^{\frac{\sigma+p-1}{p}} \right)^{p^*_{s}}\,\dx\right)^{\frac{p}{p^*_{s}}} \le \int_{\R^N } g(x) \abs{f_a(u_a)} (u_a^+(x))^{ \sigma } \, \dx.
   \end{align*}
   Taking the limit as $M \ra \infty$ and using the monotone convergence theorem, we get 
   \begin{align}\label{3.8}
       \frac{C(N,s,p) p^p}{(\sigma+p-1)^p}
       \left(\int_{\R^N}\left( (u_a^+(x))^{\frac{\sigma+p-1}{p}} \right)^{p^*_s}\,\dx\right)^{\frac{p}{p^*_s}} \le \int_{\R^N } g(x) \abs{f_a(u_a)} (u_a^+(x))^{ \sigma } \, \dx.
   \end{align}
 \noi \textbf{Step 1:}  In this step, for $\sigma = p^*_s$, we show that the set $\{ (u_a^+)^{\sigma + p -1} : a\in (0, a_1) \}$ is uniformly bounded on $L^{\frac{\sigma}{p}}(\R^N)$. From \ref{f1 new} and \ref{f1}, for any $\epsilon > 0$, there exists $C=C(\ep, a_1)$ such that $\abs{f_a(u_a)} \le C + \epsilon \abs{u_a}^{p^*_s -1}$ for all $ a \in (0, a_1)$. We can also write the previous inequality as follows: 
 \begin{align}\label{inq-e0}
     \abs{f_a(u_a)} \le C + \epsilon (u_a^+ + u_a^-)^{p^*_s -1} = C + \epsilon  \left( (u_a^+)^{p^*_s -1} + (u_a^-)^{p^*_s -1} \right).\tag{e0}
 \end{align}
Since $u_a^+$  and $u_a^-$ have disjoint supports, we estimate the right hand side of \eqref{3.8} using \eqref{inq-e0} as follows:
 \begin{align}\label{inq-e1}
     \int_{\R^N } g(x) \abs{f_a(u_a)} (u_a^+(x))^{ p^*_s } \, \dx \le C \int_{\R^N} g(x) (u_a^+(x))^{p^*_s} \, \dx + \ep   \int_{\R^N} g(x) (u_a^+(x))^{2p^*_s -1} \, \dx. \tag{e1}
 \end{align}
 First we estimate the first integral on the right hand side of \eqref{inq-e1}. We use the continuous embedding $\Dsp \hookrightarrow L^{p^*_s}(\R^N)$, the uniform boundedness of $\{ u_a : a \in (0, a_1)\}$ in $\Dsp$ (Theorem \ref{Existence and uniform boundedness}), and Remark \ref{pm part} to deduce that 
 \begin{align}\label{inq-e2}
     \int_{\R^N} g(x) (u_a^+(x))^{p^*_s} \, \dx \le C(N,s,p)\norm{g}_{\infty} \norm{u_a^+}_{s,p}^{p^*_s} \le C(N,s,p,f, g, a_1), \quad \forall \, a \in (0, a_1).\tag{e2}
 \end{align}
 In order to estimate the second integral on the right hand side of \eqref{inq-e1}, we use the H\"{o}lder's inequality with the conjugate pair $(\frac{p^*_s}{p^*_s -p}, \frac{p^*_s}{p})$ and we get
 \begin{align*}
     \int_{\R^N} g(x) (u_a^+(x))^{2p^*_s -1} \, \dx & \le \norm{g}_{\infty} \int_{\R^N} (u_a^+(x))^{p^*_s - p} (u_a^+(x))^{p^*_s+p-1} \, \dx \notag\\
     & \le \norm{g}_{\infty} \left( \int_{\R^N} (u_a^+(x))^{p^*_s} \, \dx \right)^{\frac{p^*_s - p}{p^*_s}} \left( \int_{\R^N} (u_a^+(x))^{\frac{p^*_s+p-1}{p}p^*_{s}} \, \dx \right)^{\frac{p}{p^*_s}}. 
 \end{align*}
 Now $\Dsp \hookrightarrow L^{p^*_s}(\R^N)$ and again the uniform boundedness of $\{ u_a \}$  and Remark \ref{pm part} yield
 \begin{align}\label{inq-e3}
     \int_{\R^N} g(x) (u_a^+(x))^{2p^*_s -1} \, \dx \le C(N,s,p,f,g,a_1) \left( \int_{\R^N} (u_a^+(x))^{\frac{p^*_s+p-1}{p}p^*_{s}} \, \dx \right)^{\frac{p}{p^*_s}}.\tag{e3}
 \end{align}
 Combining \eqref{3.8}, \eqref{inq-e1}, \eqref{inq-e2} and \eqref{inq-e3}, we obtain
\begin{align}\label{3.9}
\left(\int_{\R^N} (u_a^+(x))^{\frac{p^*_s+p-1}{p}p^*_{s}}  \,\dx\right)^{\frac{p}{p^*_s}}  \le \left( \frac{p^*_s+p-1}{p} \right)^p  
\left( C_1 + C_2 \ep \left( \int_{\R^N} (u_a^+(x))^{\frac{p^*_s+p-1}{p}p^*_{s}} \, \dx \right)^{\frac{p}{p^*_s}} \right),
\end{align}
where $C_1 = C_1(N,s,p,\ep,f,g, a_1)$ and $C_2 = C_2(N,s,p,f,g,a_1)$. Now we choose $\ep>0$ such that 
\begin{align*}
   \ep \left( \frac{p^*_s+p-1}{p} \right)^p C_2 < \frac{1}{2}.
\end{align*}
 Using the above choice of $\epsilon>0$, the inequality \eqref{3.9} reduces to
 \begin{align*}
     \left(\int_{\R^N} (u_a^+(x))^{\frac{p^*_s+p-1}{p}p^*_{s}} \,\dx\right)^{\frac{p}{p^*_s}} \le \left( \frac{p^*_s+p-1}{p} \right)^p C_1(N,s,p,f,g, a_1), \quad \forall \, a \in (0,a_1).
 \end{align*}
 Thus the set $\{ (u_a^+)^{p^*_s + p -1} : a\in (0, a_1) \}$ is uniformly bounded on $L^{\frac{p^*_s}{p}}(\R^N)$. \\
 \noi \textbf{Step 2:} In this step, we consider $\sigma > p^*_s$. Using $\abs{f_a(t)} \le C(f,a_1) (1 + \abs{t}^{p^*_s -1}); t \in \R$ (by \eqref{2.3}) we have 
 \begin{align*}
     \abs{f_a(u_a)} \le C(f, a_1) \left(1 +\left( (u_a^+)^{p^*_s -1} + (u_a^-)^{p^*_s -1} \right) \right),
 \end{align*}
and hence \eqref{3.8} yields  
 \begin{align}\label{3.10}
     \left(\int_{\R^N} (u_a^+(x))^{\frac{\sigma+p-1}{p}p^*_{s}} \,\dx\right)^{\frac{p}{p^*_s}}
     \le \frac{(\sigma+p-1)^p}{C(N,s,p) p^p} C(f,a_1) \int_{\R^N} g(x) \left( 1 + (u_a^+(x))^{p^*_s -1} \right) (u_a^+(x))^{\sigma} \, \dx.
 \end{align}
Set $m_1 : = \frac{p^*_s(p^*_s -1)}{\sigma -1}$ and $m_2 := \sigma - m_1$. Notice that $m_1 < p^*_s$. Applying Young's inequality with the conjugate pair $(\frac{p^*_s}{m_1}, \frac{p^*_s}{p^*_s - m_1})$ we get
 \begin{align}
     (u_a^+(x))^{\sigma} =  (u_a^+(x))^{m_1} (u_a^+(x))^{m_2} \le \frac{m_1}{p^*_s} (u_a^+(x))^{p^*_s} + \frac{p^*_s-m_1}{p^*_s} (u_a^+(x))^{\frac{p^*_s m_2}{p^*_s -m_1}},
 \end{align}
where we further observe that 
\begin{align}
    m_1 = \frac{p^*_s(p^*_s -1)}{\sigma -1} \Longleftrightarrow \frac{p^*_s m_2}{p^*_s -m_1} = p^*_s + \sigma -1.
\end{align}
 Therefore, using $\Dsp \hookrightarrow L^{p^*_s}(\R^N)$ and the uniform boundedness of $\{ u_a \}$,
 \begin{align*}
     \int_{\R^N } g(x) (u_a^+(x))^{\sigma} \, \dx & \le \norm{g}_{\infty} \left(\int_{ \R^N } (u_a^+(x))^{p^*_s} \, \dx + \int_{\R^N} (u_a^+(x))^{p^*_s + \sigma -1} \, \dx  \right) \\
     & \le C(N,s,p,f,g, a_1) \left( 1 +  \int_{\R^N} (u_a^+(x))^{p^*_s + \sigma -1} \, \dx \right).
 \end{align*}
Hence for every $\sigma > p^*_s$, \eqref{3.10} yields
 \begin{align*}
     \left(\int_{\R^N} (u_a^+(x))^{\frac{\sigma+p-1}{p}p^*_s} \,\dx\right)^{\frac{p}{p^*_s}} \le  \left( \frac{\sigma + p -1}{p}\right)^p C \left( 1 +  \int_{\R^N} (u_a^+(x))^{p^*_s + \sigma -1} \, \dx \right),
 \end{align*}
 where $C=C(N,s,p,f,g, a_1)$. From the above inequality, we get 
    \begin{align}\label{3.13}
        \left(1+\int_{\R^N} (u_a^+(x))^{\frac{\sigma+p-1}{p}p^*_{s}} \,\dx\right)^{\frac{p}{p^*_{s}}} &\le 1 + \left(\int_{\R^N } (u_a^+(x))^{\frac{\sigma+p-1}{p}p^*_{s}} \,\dx\right)^{\frac{p}{p^*_{s}}} \no \\
        &\le 1 + C \left( \frac{\sigma + p -1}{p}\right)^{p} \left(1 + \int_{\R^N} (u_a^+(x))^{p^*_s + \sigma -1} \, \dx \right) \no \\
        &\le \left( 1 + C\left( \sigma + p -1 \right)^{p} \right) \left(1 + \int_{\R^N} (u_a^+(x))^{p^*_s + \sigma -1} \, \dx \right) \no \\
        &\le \widetilde{C} \left( \sigma + p -1 \right)^{p} \left(1 + \int_{\R^N} (u_a^+(x))^{p^*_s+ \sigma -1} \, \dx \right),
    \end{align}
    where $\widetilde{C} = \widetilde{C}(N,s,p,f,g, a_1) = C + (p-1)^{-p}$.
    We consider the sequence $\{ \sigma_j \}$ such that 
    \begin{align*}
        \sigma_1 = p^*_s, \sigma_2 = 1 + \frac{p^*_s}{p}(\sigma_1 - 1), \cdots , \sigma_{j+1} = 1 + \frac{p^*_s}{p}(\sigma_j - 1). 
    \end{align*}
    Notice that $ p^*_s+ \sigma_{j+1} -1 = \frac{ \sigma_j + p -1}{p} p^*_s$ and $ \sigma_{j+1} - 1 = (\frac{p^*_s}{p})^j (\sigma_1 - 1)$. Hence using \eqref{3.13},
    we have 
    \begin{align*}
        \left(1+\int_{\R^N} (u_a^+(x))^{\frac{\sigma_{j+1}+p-1}{p}p^*_{s}} \,\dx\right)^{\frac{p}{p^*_{s}(\sigma_{j+1} - 1)}} \le & \left( \widetilde{C} \left( \sigma_{j+1} + p -1 \right)^{p} \right)^{\frac{1}{ \sigma_{j+1} -1 }} \\
        &\left( 1+\int_{\R^N} (u_a^+(x))^{\frac{\sigma_{j}+p-1}{p}p^*_{s}} \,\dx \right)^{\frac{p}{p^*_{s}(\sigma_{j} - 1)}}.
    \end{align*}
    Set $D_j := \left( 1+\int_{\R^N} (u_a^+(x))^{\frac{\sigma_{j}+p-1}{p}p^*_{s}} \,\dx \right)^{\frac{p}{p^*_{s}(\sigma_{j} - 1)}}$. Set $\eta_{j} = \sigma_{j}+p-1$. We iterate the above inequality to get 
    \begin{align*}
        D_{j+1} \le \left( {\left(\,\widetilde{C}\,\right)}^{\sum_{k=2}^{j+1} \frac{1} {\sigma_k - 1} } \left( \prod_{k = 2}^{j+1} \eta_{k}^{\frac{1}{\eta_k - p}} \right)^{p} \right) D_1.
    \end{align*}
    In view of Step 1, $D_1 \le C$ for some $C=C(N,s,p,f,g, a_1)$. Moreover, 
    \begin{align*}
        D_{j+1} \ge \left( \left(\int_{\R^N} (u_a^+(x))^{\frac{\sigma_{j+1}+p-1}{p}p^*_{s}} \, \dx \right)^{ \frac{p}{(\sigma_{j+1}+p-1)p^*_s}  } \right)^{\frac{\sigma_{j+1}+p-1}{\sigma_{j+1} - 1}} = \norm{u_a^+}_{\frac{ \sigma_{j+1}+p-1}{p}p^*_{s}}^{\frac{\eta_{j+1}}{\eta_{j+1}-p}}.
    \end{align*}
   Therefore, 
   \begin{align}\label{3.14}
       \norm{u_a^+}_{\frac{ \sigma_{j+1}+p-1}{p}p^*_{s}}^{\frac{\eta_{j+1}}{\eta_{j+1}-p}} \le \left( {\left(\,\widetilde{C}\,\right)}^{\sum_{k=2}^{j+1} \frac{1} {\eta_k - p} } \left( \prod_{k = 2}^{j+1} \eta_{k}^{\frac{1}{\eta_k - p}} \right)^{p} \right) C, \quad \forall \, a \in (0, a_1),
   \end{align}
where $C, \widetilde{C}$ are independent of $a$. Since, $\sigma_j \ra \infty$, as $j \ra \infty$, by interpolation argument we have $u_a^+ \in L^r(\R^N)$ for every $r \in [p^*_s, \infty)$, and moreover from \eqref{3.14}, $\norm{u_a^+}_{r} \le C$ for all $a \in (0, a_1)$ and $C=C(r,N,s,p,f,g,a_1)$. Further, 
\begin{align*}
  & \sum_{k=2}^\infty\frac{1}{\eta_k-p}= \frac{1}{(\eta_1-p)} \sum_{k=2}^\infty \left(\frac{p}{p_s^*} \right)^{k-1} =\frac{p}{(p_s^*-1)(p_s^*-p)} \text{ and } \\
  & \prod_{k=2}^{\infty}\eta_{k}^{\frac{1}{\eta_k - p}} = \exp\left( \sum_{k=2}^\infty \frac{\log(\eta_{k})}{\eta_{k} - p} \right) = \exp \left(\frac{p}{(p^*_{s} - p)^2} \log \left(p \left( \frac{p^*_{s}(p^*_{s}-p)}{p}  \right)^{p^*_{s}} \right) \right).
\end{align*}
Also observe that $\frac{\eta_{j+1}}{\eta_{j+1}-p} \ra 1$ as $j \ra \infty$. Therefore, taking the limit as $j \ra \infty$ in \eqref{3.14} gives $u_a^+ \in L^{\infty}(\R^N)$ and $\norm{u_a^+}_{\infty} \le C(N,s,p,f,g,a_1)$ for all $a \in (0,a_1)$. 

\noi  \textbf{Uniform bound of the negative part:} For $M > 0$, define $(u_a^-)_M= \min\{u_a^-, M\}$ for every $a \in (0, a_1)$. Clearly $(u_a^-)_M \ge 0$ and $(u_a^-)_M \in L^{\infty}(\R^N) \cap \Dsp$. For $\sigma\geq1$, we take $\phi=- ((u_a^-)_M)^{\sigma} \in \Dsp$ as a test function to get 
\begin{align*}
       \iint_{\R^{2N}} \frac{\Phi\big(u_a(x)-u_a(y)\big)\big(\phi(x) - \phi(y)\big)}{|x-y|^{N+sp}}\, \dxy =  \int_{\R^N } g(x) f_a(u_a) \phi(x) \, \dx.
   \end{align*}
Using \ref{i4}, Remark \ref{pm part}, and $\Dsp \hookrightarrow L^{p^*_s}(\R^N)$, we have 
   \begin{align*}
       & \iint_{\R^{2N}} \frac{\Phi\big(u_a(x)-u_a(y) \big) \big( (u_a^-)_M(y))^{\sigma} - ((u_a^-)_M(x))^{\sigma}\big)}{|x-y|^{N+sp}}\, \dxy \\
       &~~\ge \frac{ p^p}{(\sigma + p -1)^p} 
       \iint_{\R^{2N}} \frac{\left| ((u_a^-)_M(x))^{\frac{\sigma+p-1}{p}}-((u_a^-)_M(y))^{\frac{\sigma+p-1}{p}} \right|^p}{\abs{x-y}^{N+sp}} \, \dxy \\
       &~~~\ge \frac{C(N,s,p) p^p}{(\sigma+p-1)^p}
       \left(\int_{\R^N}\left((u_a^-)_M(x))^{\frac{\sigma+p-1}{p}} \right)^{p^*_{s}}\,\dx\right)^{\frac{p}{p^*_{s}}}.
   \end{align*}
Hence we get
   \begin{align*}
       \frac{C(N,s,p) p^p}{(\sigma+p-1)^p}
       \left(\int_{\R^N}\left((u_a^-)_M(x)^{\frac{\sigma+p-1}{p}} \right)^{p^*_{s}}\,\dx\right)^{\frac{p}{p^*_{s}}} \le \int_{\R^N } g(x) \abs{f_a(u_a)} (u_a^-(x))^{ \sigma } \, \dx.
   \end{align*}
   Taking the limit as $M \ra \infty$ and using the monotone convergence theorem, we get 
   \begin{align}\label{3.15}
       \frac{C(N,s,p) p^p}{(\sigma+p-1)^p}
       \left(\int_{\R^N}\left( (u_a^-(x))^{\frac{\sigma+p-1}{p}} \right)^{p^*_s}\,\dx\right)^{\frac{p}{p^*_s}} \le \int_{\R^N } g(x) \abs{f_a(u_a)} (u_a^-(x))^{ \sigma } \, \dx.
   \end{align}
   Now, using \eqref{3.15} and following the same procedure as above, we can conclude that for every $r \in [p^*_s, \infty]$, $u_a^- \in L^r(\R^N)$ and   $\norm{u_a^-}_{r} \le C(r,N,s,p,f,g,a_1)$ for all $a \in (0, a_1)$. 
   
   \noi Further, for $r \in [p^*_s, \infty]$, 
   \begin{align*}
       \norm{u_a}_r = \norm{u_a^+ - u_a^-}_r \le  \norm{u_a^+}_r + \norm{u_a^-}_r \le C(r,N,s,p,f,g,a_1), \quad \forall \, a \in (0, a_1).
   \end{align*}
   Thus \eqref{3.7 unifrom bound} holds. 
\end{proof}
Now, we prove that the solution $\{ u_a \}$ is uniformly bounded from below over several spaces. 
\begin{proposition}\label{uniform lower bound} 
Let $p \in (1, \frac{N}{s})$ and $f,g,a_1$ be given in Proposition \ref{regularity}. Then the following hold:
\begin{enumerate}
    \item[\rm{(i)}] There exists $C_1>0$ such that $\|u_a\|_{s,p} \geq C_1$, for all $a \in (0,a_1)$.
    \item[\rm{(ii)}] There exist $a_2 \in (0,a_1)$ and $C_2 >0$ such that $\|u_a\|_{\infty} \geq C_2$, for all $a \in (0,a_2)$.
\end{enumerate}
\end{proposition}
\begin{proof}
(i) We notice that $F_a(t) \geq -a \abs{t}$ for all $t \in \R$. For $\delta$ as given in Theorem \ref{Existence and uniform boundedness}, from \eqref{3.1} we write $I_a(u_a) \geq \delta$, for all $a \in (0,a_1)$. Using Proposition \ref{compactness weight g}, we have 
\begin{align*}
  \delta \leq I_a(u_a)  \leq \frac{1}{p}\|u_a\|_{s,p}^p + a \int_{\R^N} g(x) |u_a|\,\dx \leq \frac{1}{p}\|u_a\|_{s,p}^p + a_1 C \|u_a\|_{s,p},
\end{align*}
where $C=C(N,s,p,g)$. Thus from the above inequality, there exists $C_1=C_1(N,s,p,g,a_1,\delta)$ such that $\|u_a\|_{s,p} \geq C_1$, for all $a \in (0,a_1)$.\\
(ii) For $\delta$ as given in Theorem \ref{Existence and uniform boundedness}, from \eqref{3.1} we write $I_a(u_a) \geq \delta$ for all $a \in (0,a_1)$. Further, we have
\begin{align*}
    \frac{\|u_a\|_{s,p}^p}{p} = I_a(u_a) + \int_{\R^N} g(x) F_a(u_a)\,\dx &\geq \delta - a \int_{\R^N} g(x) \abs{u_a}\,\dx.
\end{align*}
For every $a \in (0,a_1)$, using the continuous embedding $\Dsp \hookrightarrow L^{1}(\R^N, g)$ (Proposition \ref{compactness weight g}) with embedding constant $C_2=C_2(N,s,p,g)$ and the uniform boundedness of $\{u_a\}$ in $\Dsp$ (Theorem \ref{Existence and uniform boundedness}), we obtain
\begin{align*}
    \frac{\|u_a\|_{s,p}^p}{p}\geq \delta - a C_2 \|u_a\|_{s,p} \geq \delta - a C_2 C = \delta - aC_3,
\end{align*}
where $C_3=C_2C$. Now if we  choose $a_2$ such that $0<a_2< \min\{ \frac{\delta}{C_3},a_1\}$, then
\begin{align*}
    \frac{\|u_a\|_{s,p}^p}{p}\geq \delta_0 := \delta - a_2 C_3 >0, \quad \forall \, a \in (0,a_2).
\end{align*}
Therefore, using $|f_a(u_a)| \leq C(f,a_2)(1+ |u_a|^{p_s^*-1})$ and that $u_a \in L^\infty(\R^N)$, we obtain the following estimates:
\begin{align} \label{3.16}
    \delta_0 \leq \frac{\|u_a\|_{s,p}^p}{p} = \frac{1}{p} \int_{\R^N} g(x)|f_a(u_a)u_a|\,\dx \leq C(f,a_2) \|g\|_1 (\|u_a\|_{\infty} + \|u_a\|_{\infty}^{p_s^*}).
\end{align}
Thus we can conclude from \eqref{3.16} that there exists $C_2=C_2(N,s,p,f,g,a_2,\delta)$ such that $\|u_a\|_{\infty} \geq C_2$ for all $a \in (0,a_2)$.
\end{proof}

\begin{proposition}\label{properties of limit function}
Let $p \in (1, \frac{N}{s})$ and $f,g$ be given in Proposition~\ref{regularity}. Given a sequence $a_n \ra 0$, there exists $\tilde{u} \in \Dsp$ such that $u_{a_n} \ra \tilde{u}$ in $\Dsp$ as $n \ra \infty$.  Moreover, $\tilde{u}$ is a weak solution of 
    \begin{align}\label{3.17}
       (-\Delta)_{p}^{s}u =  g(x)f_0(u) \text{ in } \mathbb{R}^{N},
    \end{align} 
     where $f_0(t) = f(t)$ for $t \ge 0$ and $f_0(t) = 0$ for $t \le 0$. Further, $\tilde{u} \in L^r(\R^N)$ for every $r \in [p^*_s, \infty]$. Furthermore, $\tilde{u} \in \mathcal{C}(\R^N)$.
\end{proposition}

\begin{proof}
    Since $a_n \ra 0$ as $n \ra \infty$, there exists $n_1 \in \N$ such that $a_n < a_2$ for all $n \ge n_1$. For each $n \ge n_1$, $u_{a_n} \in \Dsp$ is a critical point of $I_{a_n}$. Moreover, by Theorem \ref{Existence and uniform boundedness}, the following hold up to a subsequence: 
    \begin{align*}
        I_{a_n}(u_{a_n}) \ra c \text{ in } \R \text{ as } n \ra \infty, I'_{a_n}(u_{a_n}) = 0 \text{ and } \norm{u_{a_n}}_{s,p} \le C \quad \forall \, n \ge n_1.
    \end{align*}
    Therefore, $\{ u_{a_n}\}$ is a bounded P-S sequence, and hence up to subsequence $u_{a_n} \ra \tilde{u}$ in $\Dsp$. This implies $u_{a_n} \ra \tilde{u}$ in $L^{p^*_s}(\R^N)$ and up to a further subsequence $u_{a_n}(x) \ra \tilde{u}(x)$ a.e. in $\R^N$. Now for every $\phi \in \Dsp$ with $\phi \ge 0$ we show that 
    \begin{align*}
        \int_{\R^N } g(x) f_{a_n}(u_{a_n}) \phi(x) \, \dx \ra \int_{\R^N } g(x) f_0(\tilde{u}) \phi(x) \, \dx, \text{ as } n \ra \infty.
    \end{align*}
    We split 
    \begin{align*}
        \abs{f_{a_n}(u_{a_n}) - f_0(\tilde{u})} \le \abs{f_{a_n}(u_{a_n}) - f_0(u_{a_n})} + \abs{ f_0(u_{a_n}) - f_0(\tilde{u})}.
    \end{align*}
    Using the continuity of $K_0'$,  $\int_{\R^N} g(x) \abs{ f_0(u_{a_n}) - f_0(\tilde{u})} \phi(x) \ra 0$, as $n \ra \infty$. Further, $\abs{f_{a_n}(u_{a_n}) - f_0(u_{a_n})} \le a_n$. Therefore, 
    \begin{align*}
        &\int_{\R^N} g(x) \abs{f_{a_n}(u_{a_n}) - f_0(\tilde{u})} \phi(x) \, \dx \\
        &\le a_n \int_{\R^N} g(x) \phi(x) \, \dx +  \int_{\R^N} g(x) \abs{ f_0(u_{a_n}) - f_0(\tilde{u})} \phi(x) \, \dx \ra 0, \text{ as } n \ra \infty.
    \end{align*}
    From the weak formulation 
    \begin{align*}
        \iint_{\R^{2N}} \frac{\Phi\big(u_{a_n}(x)-u_{a_n}(y)\big)\big(\phi(x) - \phi(y)\big)}{|x-y|^{N+sp}}\, \dxy = \int_{\R^N } g(x) f_{a_n}(u_{a_n}) \phi(x) \, \dx.
    \end{align*}
    Taking the limit as $n \ra \infty$ gives 
    \begin{align*}
        \iint_{\R^{2N}} \frac{\Phi\big(\tilde{u}(x)-\tilde{u}(y)\big)\big(\phi(x) - \phi(y)\big)}{|x-y|^{N+sp}}\, \dxy = \int_{\R^N } g(x) f_0(\tilde{u}) \phi(x) \, \dx, ~~ \forall \, \phi \in \Dsp, \phi \ge 0.
    \end{align*}
    Now, for any $\phi \in \Dsp$, we write $\phi = \phi^+ - \phi^-$ where we see that the above identity holds for both $\phi^+$ and $\phi^-$. Thus 
    \begin{align}\label{3.18}
        \iint_{\R^{2N}} \frac{\Phi\big(\tilde{u}(x)-\tilde{u}(y)\big)\big(\phi(x) - \phi(y)\big)}{|x-y|^{N+sp}}\, \dxy = \int_{\R^N } g(x) f_0(\tilde{u}) \phi(x) \, \dx, ~~ \forall \, \phi \in \Dsp, 
    \end{align}
   which implies that $\tilde{u}$ is a weak solution to \eqref{3.17}. Now using a similar set of arguments as in the proof of Proposition \ref{regularity}, we obtain  $\tilde{u} \in L^r(\R^N)$ for every $r \in [p^*_s, \infty]$. Next, we show that $\tilde{u} \in \mathcal{C}(\R^N)$.
    Let $\Omega \subset \R^N$ be an open and bounded Lipschitz set. Let $ \phi \in W^{s,p}(\Omega)$ and  $\phi$ is compactly supported in $\Omega$. Define $\tilde{\phi}(x) = \phi(x)$ for $x \in \Omega$, and $\tilde{\phi}(x) = 0$ for $x \in \R^N \setminus \Omega$. According to \cite[Theorem 5.1]{DNPV2012}, $\tilde{\phi} \in W^{s,p}(\R^N)$. Using $W^{s,p}(\R^N) \hookrightarrow L^{p_s^*}(\R^N)$\cite[Theorem 6.5]{DNPV2012}, it follows that $\tilde{\phi} \in L^{p_s^*} (\R^N)$. Therefore, based on the characterization of $\Dsp$ in \eqref{dsp}, we  conclude that $\tilde{\phi} \in \Dsp$.
    Now $\tilde{u} \in \Dsp \cap L^{\infty}(\R^N)$. Using the embeddings $\Dsp \hookrightarrow L^{p^*_s}(\R^N)$ and $L^{p^*_s}(\Omega) \hookrightarrow L^p(\Omega)$, observe that $\Dsp \subset W^{s,p}_{\text{loc}}(\Omega)$. Also from the definition, $L^{\infty}(\R^N) \subset L^{p-1}_{sp}(\R^N)$. Therefore, in view of \eqref{3.18} we see that $\tilde{u} \in W^{s,p}_{\text{loc}}(\Omega) \cap L^{p-1}_{sp}(\R^N)$ satisfies the following identity:
\begin{align*}
    \iint_{\R^{2N}} \frac{\Phi\big(\tilde{u}(x)-\tilde{u}(y)\big)\big(\tilde{\phi}(x) - \tilde{\phi}(y)\big)}{|x-y|^{N+sp}}\, \dxy & = \int_{\R^N }  g(x) f_0(\tilde{u}) \tilde{\phi}(x) \, \dx, \\
    & = \int_{ \Omega } g(x) f_0(\tilde{u}) \phi(x) \, \dx, ~~ \forall \phi \in W^{s,p}(\Omega), \text{supp}(\phi) \subset \Omega.
\end{align*}
Moreover, $g \in L^{1}(\R^N) \cap L^{\infty}(\R^N)$, $\abs{f_0(\tilde{u})} \leq \epsilon \abs{\tilde{u}}^{p-1} + C(f,\epsilon) \abs{\tilde{u}}^{\gamma-1}$ (by \eqref{2.2}), and $\tilde{u} \in L^{\infty}(\R^N)$. Hence,  $g(x)f_0(\tilde{u}) \in L^{q}(\R^N)$ for $q>\frac{N}{sp}$. Thus, applying \cite[Theorem 1.4]{Brasco_holder} for $p\geq 2$ and \cite[Theorem 1.2]{Garain2023} for $1<p<2$ over $\tilde{u}$, where 
\[ (-\Delta)_{p}^{s} \tilde{u} = g(x) f_0(\tilde{u}) \text{ in } \Omega,\]
we conclude that $\tilde{u} \in \mathcal{C}_{\text{loc}}^{\delta}(\Omega)$ for some $\delta \in (0,1)$. In particular, $\tilde{u} \in \mathcal{C}_{\text{loc}}(\Omega)$ and hence we can get $\tilde{u} \in \mathcal{C}(\Omega)$ for every bounded open Lipschitz set $\Omega \subset \R^N$. Next, for any compact set $K \subset \R^N$, we have $K \subset \Omega$ for some bounded open Lipschitz set $\Omega$. Thus, $\tilde{u} \in \mathcal{C}(K)$ as well which implies $\tilde{u} \in \mathcal{C}_{\text{loc}}(\R^N)$. Therefore,  $\tilde{u} \in \mathcal{C}(\R^N)$. 
\end{proof}

Next, for a sequence $\{ a_n \}$ going to zero, we prove the uniform convergence of $\{ u_{a_n} \}$ over $\R^N$.

\begin{proposition}\label{unifrom convergence}
     Let $p \in (\frac{2N}{N+2s}, \frac{N}{s})$. Consider $a_n, \tilde{u}$ given in Proposition \ref{properties of limit function}. Then $$\norm{ u_{a_n} - \tilde{u}}_{\infty} \ra 0, \text{ as } n \ra \infty.$$
\end{proposition}

\begin{proof}
  Using $a_n \ra 0$, there exists $n_1 \in \N$ such that $a_n < a_2$ for all $n \ge n_1$. Now for $n \ge n_1$, $I_{a_n}'(u_{a_n}) = 0$, and hence the following identity holds:
  \begin{align*}
     \iint_{\R^{2N}} \frac{\Phi(u_{a_n}(x)-u_{a_n}(y))(\phi(x)-\phi(y))}{|x-y|^{N+sp}}\, \dxy = \int_{\R^N} g(x) f_{{a_n}}(u_{a_n})\phi(x) \,\dx, \quad \forall \, \phi \in \Dsp.
  \end{align*}
  For brevity we denote $u_n: = u_{a_n}$ and $f_n(\cdot) := f_{a_n}(\cdot)$. Applying Proposition \ref{properties of limit function}, $u_n \ra \tilde{u}$ in $\Dsp$ and 
  \begin{align*}
     \iint_{\R^{2N}} \frac{\Phi(\tilde{u}(x)-\tilde{u}(y) )(\phi(x)-\phi(y))}{|x-y|^{N+sp}}\, \dxy = \int_{\R^N} g(x) f_{0}(\tilde{u})\phi(x) \,\dx, \quad \forall \, \phi \in \Dsp.
  \end{align*}
  Subtracting the above identities, we get 
  \begin{align}\label{3.19}
      &\iint_{\R^{2N}} \frac{ \Phi(u_{n}(x)-u_{n}(y)) - \Phi(\tilde{u}(x)-\tilde{u}(y) )}{|x-y|^{N+sp}} (\phi(x)-\phi(y)) \, \dxy \no \\ 
      & = \int_{\R^N} g(x) \left( f_n(u_n) - f_0(\tilde{u}) \right) \phi(x) \, \dx. 
  \end{align}
Now we define $w_n = u_n - \tilde{u}$. Since $u_n, \tilde{u} \in \Dsp \cap L^r(\R^N)$ (Proposition \ref{regularity} and Proposition \ref{properties of limit function}), we have $w_n \in \Dsp \cap L^r(\R^N)$ for every $r \in [p^*_s, \infty]$. Taking $\alpha =w_n(x)$ and $\beta=w_n(y)$ into the inequality \eqref{inequality-squassina} and following the same arguments in the proof of Proposition \ref{regularity}, where it is shown that $\phi=((u_a^+)_M)^{\sigma} \in \Dsp$, we get $\abs{w_n}^{q-1} w_n \in \Dsp$ for every $q \ge 1$. 
   
\noi \textbf{For $p \ge 2$:} We use \cite[Lemma 2.3]{Antonio2020} to obtain the following lower bound of the right hand side of \eqref{3.19}:
\begin{align*}
\norm{\abs{w_n}^{\frac{q+p-1}{p}-1}w_n}_{s,p}^{p} \le C(p) q^{p-1} & \iint_{\R^{2N}} \frac{ \Phi(u_{n}(x)-u_{n}(y)) - \Phi(\tilde{u}(x)-\tilde{u}(y) )}{|x-y|^{N+sp}} \\
&\left( \abs{w_n(x)}^{q-1} w_n(x)-(\abs{w_n(y)}^{q-1} w_n(y) \right) \, \dxy. 
\end{align*}
Hence using the embedding $\Dsp \hookrightarrow L^{p^*_s}(\R^N)$  and choosing $\phi=\abs{w_n}^{q-1}w_n$ in \eqref{3.19} we obtain 
\begin{align}\label{3.20}
  & C(N,s,p)\left(\int_{\R^N}\left( \abs{w_n(x)}^{\frac{q+p-1}{p}} \right)^{p^*_s}\,\dx\right)^{\frac{p}{p^*_s}} \le  \norm{\abs{w_n}^{\frac{q+p-1}{p}-1}w_n}_{s,p}^{p} \no \\
  &\le C(p)q^{p-1} \int_{\R^N} g(x) \left| f_n(u_n) - f_0(\tilde{u}) \right| |w_n(x)|^q \, \dx.
\end{align}
Define $g_n(t):= f_n(t+ \tilde{u}) - f_0(\tilde{u}).$ Using \eqref{2.3}, we have $\abs{f_n(t)} \le C(f,a_2) (1 + \abs{t}^{p^*_s -1}), \abs{f_0(t)} \le C(f)(1 + \abs{t}^{p^*_s -1}); t \in \R$. By noting that $g_n(w_n) = f_n(u_n) - f_0(\tilde{u})$, 
\begin{align}\label{3.21}
    \abs{g_n(w_n)} \le \abs{f_n(u_n)} + \abs{f_0(\tilde{u})} \le C(f, a_2) \left(1+ \abs{u_n}^{p^*_s -1} + \abs{\tilde{u}}^{p^*_s -1}\right) \le C \left(1 + 
  \abs{w_n}^{p^*_s -1} \right),
\end{align}
where $C$ does not depend on $a_n$. Therefore, \eqref{3.20} yields
\begin{align}\label{3.22}
     \left(\int_{\R^N}\left( \abs{w_n(x)}^{\frac{q+p-1}{p}} \right)^{p^*_s}\,\dx\right)^{\frac{p}{p^*_s}} & \le C(N,s,p,f,a_2) q^{p-1} \no \\
     &\int_{\R^N} g(x)\left(1 + \abs{w_n(x)}^{p^*_s -1} \right) \abs{w_n(x)}^q \, \dx. 
\end{align}
Let us consider
\begin{align}\label{3.25}
    \Bar{q} > \frac{N(p_s^* -1)}{sp} \; \text{ and } \; \Bar{r} = \frac{\Bar{q}}{p^*_s-1}.
\end{align} 
Observe that $\frac{N(p_s^* -1)}{sp}>p_s^*$. Hence $\{w_n\}$ is bounded in $L^{\Bar{q}}(\R^N)$ (by Proposition \ref{regularity} and Proposition \ref{properties of limit function}). 
For the conjugate pair $(\Bar{r},\Bar{r}')$, define 
\begin{align}\label{GnHn}
    G_n(x,w_n) := (g(x))^{\frac{1}{\Bar{r}}}\left( 1+ \abs{w_n(x)}^{p^*_s-1}\right) \, \text{ and } \, H_n(x,w_n):=(g(x))^{\frac{1}{\Bar{r}'}}\abs{w_n(x)}^{q}.
\end{align}
Since $g \in L^{1}(\R^N)\cap L^{\infty}(\R^N)$, we have 
\begin{align*}
    \int_{\R^N} \abs{G_n(x,w_n)}^{\Bar{r}}\,\dx \leq 2^{\Bar{r}-1} \int_{\R^N} g(x)\,\dx + 2^{\Bar{r}-1}\norm{g}_{\infty} \int_{\R^N}\abs{w_n(x)}^{\Bar{q}}\,\dx \leq C(N,s,p,f,g,a_2).
\end{align*}
 If $q \Bar{r}' <p_s^*$, then applying Proposition \ref{compactness weight g},   
 $\int_{\R^N} \abs{H_n(x,w_n)}^{\Bar{r}'}\,\dx = \int_{\R^N}g(x)\abs{w_n(x)}^{q \Bar{r}'}\,\dx < \infty$. If $q \Bar{r}' \geq p_s^*$, then using boundedness of $\{w_n\}$ in $L^{q \Bar{r}'}(\R^N)$ (by Proposition \ref{regularity} and Proposition \ref{properties of limit function}), $\int_{\R^N} \abs{H_n(x,w_n)}^{\Bar{r}'}\,\dx \leq \norm{g}_{\infty}  \int_{\R^N}\abs{w_n(x)}^{q \Bar{r}'}\,\dx <\infty.$
  Hence we apply the H\"{o}lder's inequality with the conjugate pair $(\Bar{r}, \Bar{r}')$ to get from \eqref{3.22} that
\begin{align}\label{3.30}
     &\left(\int_{\R^N} \abs{w_n(x)}^{\frac{(q+p-1)p^*_s}{p}} \,\dx\right)^{\frac{p}{p^*_s}} \le C q^{p-1} \int_{\R^N} G_n(x,w_n) H_n(x,w_n) \, \dx \no \\
     & \le C q^{p-1} \left(\int_{\R^N}\abs{G_n(x,w_n)}^{\Bar{r}}\,\dx \right)^{\frac{1}{\Bar{r}}} \left( \int_{\R^N}\abs{H_n(x,w_n)}^{\Bar{r}'}\,\dx \right)^{\frac{1}{\Bar{r}'}} \no \\
     &\leq \widetilde{C} q^{p-1} \left( \int_{\R^N} g(x)\abs{w_n(x)}^{q\Bar{r}'}\,\dx \right)^{\frac{1}{\Bar{r}'}},
\end{align}
where $C= C(N,s,p,f,a_2)$ and $\widetilde{C}= \widetilde{C}(N,s,p,f,g,a_2)$. Define $\theta := \frac{p^*_s}{p \Bar{r}'}.$ Using the definition of $\Bar{r}$ in \eqref{3.25} it follows that $\theta >1$.
We consider two sequences $\{l_j\}$ and $\{m_j\}$ by the following recursive process:
\[ l_0 = p_s^*,  \quad l_{j+1}= \theta l_j + \frac{p^*_s(p-1)}{p}, \quad m_j = \frac{l_j}{\Bar{r}'}.\]
Observe that $l_j \ge p^*_s$ and $l_{j+1}= \frac{(m_j+p-1)p^*_s}{p}$ for each $j \in \mathbb{N}$, and $l_j,m_j \ra \infty$ as $j \rightarrow \infty$. Thus, for $q=m_j$ in \eqref{3.30} we obtain the following for all $n, j \in \mathbb{N}$:
\begin{align}\label{3.31}
     \int_{\R^N} \abs{w_n(x)}^{l_{j+1}} \,\dx 
      &\le \left(\widetilde{C} m_{j}^{p-1}\right)^{\frac{p^*_s}{p}} \left( \int_{\R^N} g(x)\abs{w_n(x)}^{m_j\Bar{r}'}\,\dx \right)^{\frac{p^*_s}{p \Bar{r}'}} \no \\
      &\le \left(\widetilde{C} m_{j}^{p}\right)^{\theta \Bar{r}'} \norm{g}_{\infty}^{\theta} \left( \int_{\R^N}\abs{w_n(x)}^{l_j}\,\dx \right)^{\theta} \no\\ 
      &\le \left(\widetilde{C} m_{j}^{p}\right)^{\theta \Bar{r}'} \left( \int_{\R^N}\abs{w_n(x)}^{l_j}\,\dx \right)^{\theta},
\end{align}
for some $\widetilde{C}= \widetilde{C}(N,s,p,f,g,a_2)$. By Lemma \ref{appendix lem1} we observe that $m_j \sim \frac{\theta^j}{\Bar{r}'}$ as $j \ra \infty$. Thus iterating on \eqref{3.31} we obtain
\begin{align}\label{3.32}
    \int_{\R^N} \abs{w_n(x)}^{l_{j}} \,\dx & \leq \prod_{i=0}^{j-1} \left(\widetilde{C} m_{i}^{p} \right)^{\theta^{j-i} \Bar{r}'} \left(\int_{\R^N}\abs{w_n(x)}^{l_0}\,\dx \right)^{\theta^j} \no \\
    & \leq  \left(\widetilde{C} \right)^{\sum\limits_{i=0}^{j-1} \theta^{j-i} \Bar{r}'} \theta^{p \Bar{r}' \sum\limits_{i=0}^{j-1} i \theta^{j-i} } \left(\int_{\R^N}\abs{w_n(x)}^{p^*_s}\,\dx \right)^{\theta^j}.
\end{align}
Notice that $S_1:= \sum\limits_{i=0}^{\infty}\theta^{-i} < \infty$ and $S_2 := \sum\limits_{i=0}^{\infty} i \theta^{-i} < \infty$. Take $C> \max \{1, \widetilde{C} \}$. Then for all $n \ge n_1,j \in \N$ using \eqref{3.32} we deduce 
\begin{align*}
    \int_{\R^N} \abs{w_n(x)}^{l_{j}} \,\dx  \leq  C^{\theta^j S_1} \theta^{p \Bar{r}' \theta^{j} S_2} \left(\int_{\R^N}\abs{w_n(x)}^{p^*_s}\,\dx \right)^{\theta^j}.
\end{align*}
Since $w_n \ra 0$ in $L^{p^*_s}(\R^N)$, there exists $n_2 \in \N$ such that $\norm{w_n}_{p^*_s}^{p^*_s} <1$ for all $n \ge n_2$. Therefore, for $n \ge \max\{n_1,n_2\}$ from the above inequality, we obtain
\begin{align*}
     \norm{w_n}_{l_{j}} \leq  C^{\frac{\theta^j}{l_j} S_1} \theta^{p \Bar{r}' \frac{\theta^j}{l_j} S_2} \left( \norm{w_n}_{p^*_s}^{p^*_s} \right)^{\frac{\theta^j}{l_j}} \leq C^{\beta_2 S_1} \theta^{p \Bar{r}' \beta_2 S_2} \left( \norm{w_n}_{p^*_s}^{p^*_s} \right)^{\beta_1},
\end{align*}
where $\beta_1,\beta_2>0$ have been chosen such that for all $j \in \N$, $\beta_1 < \frac{\theta^j}{l_j} <\beta_2$ (see Lemma \ref{appendix lem1}). 
Therefore, there exists $C>1$ such that for all $n \ge \max\{n_1,n_2\}$ and $j \in \N$ large enough
\begin{equation} \label{3.33}
    \norm{w_n}_{l_{j}} \leq  C \norm{w_n}_{p^*_s}^{\beta_1 p^*_s}
\end{equation}
Therefore, taking the limit as $j \ra \infty$ in \eqref{3.33} and recalling that $w_n \in L^{\infty}(\R^N)$, we obtain for each $n \ge \max\{n_1,n_2\}$,
\begin{equation*}
    \norm{w_n}_{\infty} \leq  C \norm{w_n}_{p^*_s}^{\beta_1 p^*_s},
\end{equation*}
where $C = C(N,s,p,f,g,a_2)$. Finally, we take the limit as $n \ra \infty$ in the above inequality and use $w_n \ra 0$ in $L^{p^*_s}(\R^N)$ to get the required
convergence.

\noi \textbf{For $\frac{2N}{N+2s}<p<2$:} For any $q \ge 1$, we take $\abs{w_n}^{q -1} w_n \in \Dsp$ as a test function in \eqref{3.19} and use \cite[Lemma 2.4]{Antonio2020} to obtain 
\begin{align*}
 \frac{\norm{ \abs{w_n}^{\frac{q+1}{2}-1} w_n}_{s,p}^2}{ \left( \norm{u_n}_{s,p}^p + \norm{\tilde{u}}_{s,p}^p \right)^{2-p}} &\le Cq \iint_{\R^{2N}} \frac{ \Phi(u_{n}(x)-u_{n}(y)) - \Phi(\tilde{u}(x)-\tilde{u}(y) )}{|x-y|^{N+sp}} \\
&\quad \quad \quad \left( \abs{w_n(x)}^{q-1} w_n(x)-(\abs{w_n(y)}^{q-1} w_n(y) \right) \, \dxy \\
&\le Cq \int_{\R^N} g(x) \left| f_n(u_n) - f_0(\tilde{u}) \right| |w_n(x)|^q \, \dx.
\end{align*}
Hence using $\Dsp \hookrightarrow L^{p^*_s}(\R^N)$, the uniform boundedness of $\{ u_n\}$ over $\Dsp$ (Theorem \ref{Existence and uniform boundedness}), and \eqref{3.21}  we get
\begin{align}\label{3.34}
    &\left(\int_{\R^N}\left( \abs{w_n(x)}^{\frac{q+1}{2}} \right)^{p^*_s}\,\dx\right)^{\frac{2}{p^*_s}} \le q C(N,s,p) \left( \norm{u_n}_{s,p}^p + \norm{\tilde{u}}_{s,p}^p \right)^{2-p} \no \\
    & \quad \quad \quad \int_{\R^N} g(x) \left| f_n(u_n) - f_0(\tilde{u}) \right| |w_n(x)|^q \, \dx \no \\
    & \le q C(N,s,p,f,g,a_2) \int_{\R^N} g(x)  \left(1 + \abs{w_n(x)}^{p^*_s -1} \right) \abs{w_n(x)}^q \, \dx. 
\end{align}
Observe that $p>\frac{2N}{N+2s} \iff p_s^*>2$. For $\Bar{q}$ as given in \eqref{3.25}, we set
\begin{align}\label{u1}
     \Bar{r} = \frac{\Bar{q}}{p^*_s-2}.
\end{align} 
For this $\Bar{r}$ and its conjugate exponent $\Bar{r}'$, we consider $G_n$ and $H_n$ as defined in \eqref{GnHn}.
Using the fact that $\Bar{r}(p_s^* -1) > \Bar{q}>p_s^*$ and boundedness of $\{w_n\}$ in $L^{\Bar{r}(p_s^* -1)}(\R^N)$, 
\begin{align*}
    \int_{\R^N} \abs{G_n(x,w_n)}^{\Bar{r}}\,\dx \leq 2^{\Bar{r}-1} \int_{\R^N} g(x)\,\dx + 2^{\Bar{r}-1}\norm{g}_{\infty} \int_{\R^N}\abs{w_n(x)}^{\Bar{r}(p_s^* -1)}\,\dx \leq C(N,s,p,f,g,a_2).
\end{align*}
Moreover, following similar arguments as given earlier, $\int_{\R^N} \abs{H_n(x,w_n)}^{\Bar{r}'}\,\dx< \infty$. Thus applying the H\"{o}lder's inequality with the conjugate pair $(\Bar{r}, \Bar{r}')$ we get from \eqref{3.34} that
\begin{align}\label{3.39}
     &\left(\int_{\R^N} \abs{w_n(x)}^{\frac{(q+1)p^*_s}{2}} \,\dx\right)^{\frac{2}{p^*_s}} \le C q \int_{\R^N} G_n(x,w_n) H_n(x,w_n) \, \dx \no \\
     & \le C q \left(\int_{\R^N}\abs{G_n(x,w_n)}^{\Bar{r}}\,\dx \right)^{\frac{1}{\Bar{r}}} \left( \int_{\R^N} g(x) \abs{w_n(x)}^{q\Bar{r}'}\,\dx \right)^{\frac{1}{\Bar{r}'}} \no \\
     &\leq \widetilde{C} q \left( \int_{\R^N} g(x) \abs{w_n(x)}^{q\Bar{r}'}\,\dx \right)^{\frac{1}{\Bar{r}'}},
\end{align}
where $\widetilde{C}= \widetilde{C}(N,s,p,f,g,a_2)$. Define $\theta := \frac{p^*_s}{2 \Bar{r}'}.$ Using the definition of $\Bar{r}$ in \eqref{u1} we can verify that  $\theta >1$.
We consider two sequences $\{l_j\}$ and $\{m_j\}$ by the following recursive process:
\[ l_0 = p_s^*,  \quad l_{j+1}= \theta l_j + \frac{p^*_s}{2}, \quad m_j = \frac{l_j}{\Bar{r}'}.\]
Observe that $l_{j+1}= \frac{(m_j+1)p^*_s}{2}$  and $l_j,m_j \ra \infty$ as $j \rightarrow \infty$. Thus, for $q=m_j$ in \eqref{3.39} we obtain the following for all $n, j \in \mathbb{N}$:
\begin{align}\label{3.40}
     \int_{\R^N} \abs{w_n(x)}^{l_{j+1}} \,\dx 
      &\le  \left(\widetilde{C} m_{j}\right)^{\frac{p^*_s}{2}} \left( \int_{\R^N} g(x)\abs{w_n(x)}^{m_j\Bar{r}'}\,\dx \right)^{\frac{p^*_s}{2 \Bar{r}'}} \no \\
      & \le \left(\widetilde{C} m_{j}^{2}\right)^{\theta \Bar{r}'} \norm{g}_{\infty}^{\theta}  \left( \int_{\R^N}\abs{w_n(x)}^{l_j}\,\dx \right)^{\theta} \no \\
      &\le \left(\widetilde{C} m_{j}^{2}\right)^{\theta \Bar{r}'} \left( \int_{\R^N}\abs{w_n(x)}^{l_j}\,\dx \right)^{\theta},
\end{align}
for some $\widetilde{C}= \widetilde{C}(N,s,p,f,g,a_2)$. Iterating  \eqref{3.40} and following a similar calculation, we can deduce that for all $n \geq n_1,j \in \N$
\begin{align*}
    \int_{\R^N} \abs{w_n(x)}^{l_{j}} \,\dx  \leq  C^{\theta^jS_1 }\theta^{2 \Bar{r}' \theta^{j}S_2 } \left(\int_{\R^N}\abs{w_n(x)}^{p^*_s}\,\dx \right)^{\theta^j},
\end{align*}
where $C> \max \{1, \widetilde{C} \}$. Since $w_n \ra 0$ in $L^{p^*_s}(\R^N)$, there exists $n_3 \in \N$ such that $\norm{w_n}_{p^*_s}^{p^*_s} <1$ for all $n \ge n_3$. Therefore, for $n \ge \max\{n_1,n_3\}$ we obtain
\begin{align*}
     \norm{w_n}_{l_{j}}  \leq  C^{\frac{\theta^j}{l_j}S_1} \theta^{2 \Bar{r}' \frac{\theta^j}{l_j} S_2} \left( \norm{w_n}_{p^*_s}^{p^*_s} \right)^{\frac{\theta^j}{l_j}} \leq C^{\alpha_2 S_1} \theta^{2 \Bar{r}'  \alpha_2 S_2} \left( \norm{w_n}_{p^*_s}^{p^*_s} \right)^{\alpha_1},
\end{align*}
where $\alpha_1,\alpha_2$ have been chosen such that for all $j \in \N$, $\alpha_1 < \frac{\theta^j}{l_j} <\alpha_2$ (see Lemma \ref{appendix lem1}).
Finally, we find a $C>1$ such that for all $n \ge \max\{n_1,n_3\}$ and $j \in \N$ large enough
\begin{equation} \label{3.41}
    \norm{w_n}_{l_{j}} \leq  C \norm{w_n}_{p^*_s}^{\alpha_1 p^*_s}.
\end{equation}
Therefore, taking the limit as $j \ra \infty$ in \eqref{3.41} we obtain for each $n \ge \max\{n_1,n_3\}$,
\begin{equation*}
    \norm{w_n}_{\infty} \leq  C \norm{w_n}_{p^*_s}^{\alpha_1 p^*_s},
\end{equation*}
where $C = C(N,s,p,f,g,a_2)$. Taking the limit as $n \ra \infty$ in the above inequality, we get the required convergence.
\end{proof}

In the following proposition, we state a strong maximum principle for a nonlocal equation defined on $\R^N$. Our proof follows using the similar arguments given in \cite[Theorem A.1]{Brasco_convexity_properties}. For the sake of completeness, we sketch the proof. 
\begin{proposition}[Strong Maximum Principle]\label{SMP}
   Let $p \in (1, \infty)$ and $s \in (0,1)$. Let $f,g$ be given in Theorem \ref{Existence and uniform boundedness}. Assume that $u \in \Dsp$ is a weak supersolution of the following equation:
   \begin{align*}
       (-\Delta)_{p}^{s}u =  g(x)f(u)  \text{ in } \mathbb{R}^{N},
\end{align*}
and $u \ge 0$ a.e. in $\R^N$. Then either $u \equiv 0$ or $u>0$ a.e. in $\R^N$.
\end{proposition}

\begin{proof}
    Take $\phi \in \Dsp$ and $\phi \ge 0$. By the hypothesis, we have 
    \begin{align*}
        \iint_{\R^{2N}} \frac{\Phi\big(u(x)-u(y)\big)\big(\phi(x) - \phi(y)\big)}{|x-y|^{N+sp}}\, \dxy \ge  \int_{\R^N } g(x) f(u) \phi(x) \, \dx \ge 0.
    \end{align*}
     Let $K \subset \subset \R^N$ be any compact connected set. We first show that either $u \equiv 0$ or $u>0$ a.e. in $K$. Since $K$ is compact, we choose $x_1,x_2,\dots, x_k$ in $\R^N$ such that $K \subset \cup_{i=1}^k B_{\frac{r}{2}}(x_i)$, and $\abs{B_{\frac{r}{2}}(x_i) \cap B_{\frac{r}{2}}(x_{i+1})} >0$ for each $i$. Suppose $u \equiv 0$ on a subset of $K$ with a positive measure. Then there exists $j \in \{1, \dots, k \}$ such that $\mathcal{A} = \{ x \in B_{\frac{r}{2}}(x_j) : u(x) = 0 \}$ has a positive measure, i.e., $|\mathcal{A}| >0$.
     We define 
    \begin{align*}
        F_{\delta}(x) = \log \left( 1 + \frac{u(x)}{\delta} \right) \text{ for } x \in B_{\frac{r}{2}}(x_j) \text{ and } \delta>0.
    \end{align*}
    Clearly, $F_{\delta} \equiv 0$ on $\mathcal{A}$. Take $x \in B_{\frac{r}{2}}(x_j)$ and $y \in \mathcal{A}$ with $y \neq x$. Then 
    \begin{align*}
        \abs{F_{\delta}(x)}^p = \frac{ \abs {F_{\delta}(x) - F_{\delta}(y)}^p }{ \abs{x-y}^{N+sp} } \abs{x-y}^{N+sp}.
     \end{align*}
     Integrating the above identity over $\mathcal{A}$ we get
  \begin{align*}
      \abs{ \mathcal{A} } \abs{F_{\delta}(x)}^p \le \max_{x,y \in B_{\frac{r}{2}}(x_j)}{ \abs{x-y}^{N+sp} } \int_{ B_{\frac{r}{2}}(x_j)  } \left|\log \left( \frac{u(x) + \delta}{u(y) + \delta} \right) \right|^p \, \frac{\dy}{\abs{x-y}^{N+sp}}.
  \end{align*}
Further, the integration over $ B_{\frac{r}{2}}(x_j) $ yields
 \begin{align}\label{3.42}
     \int_{ B_{\frac{r}{2}}(x_j) } \abs{F_{\delta}(x)}^p \, \dx \le \frac{r^{N+sp}}{\abs{\mathcal{A}}} \iint_{B_{\frac{r}{2}}(x_j) \times B_{\frac{r}{2}}(x_j)} \left|\log \left( \frac{u(x) + \delta}{u(y) + \delta} \right) \right|^p \, \frac{\dy \dx}{\abs{x-y}^{N+sp}}.
 \end{align}
 Now on $B_{\frac{r}{2}}(x_j)$ we use the Logarithmic energy estimate on $u$ (see \cite[Lemma 1.3]{Castro_local-Behaviour}), to get 
 \begin{align}\label{3.43}
     \iint_{B_{\frac{r}{2}}(x_j) \times B_{\frac{r}{2}}(x_j)} \left|\log \left( \frac{u(x) + \delta}{u(y) + \delta} \right) \right|^p \, \frac{\dy \dx}{\abs{x-y}^{N+sp}}
     \leq C(N,s,p) r^{N-sp}.
 \end{align}
From \eqref{3.42} and \eqref{3.43} we get for every $\delta > 0$ that
 \begin{align*}
     \int_{ B_{\frac{r}{2}}(x_j) } \left|  \log \left( 1 + \frac{u(x)}{\delta} \right)  \right|^p \, \dx  \le C(N,s,p) \frac{ r^{2N} }{ \abs{\mathcal{A}}}.
 \end{align*}
 Taking $\delta \ra 0$ in the above estimate, we get $u \equiv 0$ a.e. in $B_{\frac{r}{2}}(x_j)$. Moreover, $u$ is identically zero on a subset of positive measure in $B_{\frac{r}{2}}(x_{j+1})$ since $\abs{B_{\frac{r}{2}}(x_j) \cap B_{\frac{r}{2}}(x_{j+1})} >0$. Consequently, repeating the same arguments, we obtain $u \equiv 0$ a.e. in $B_{\frac{r}{2}}(x_{j+1})$, and then $u \equiv 0$ a.e. in $B_{\frac{r}{2}}(x_{i})$ for every $i = 1, \dots, k$. Thus $u \equiv 0$ a.e. in $K$. So for every relatively compact set $K$ in $\R^N$, either $u \equiv 0$ or $u>0$ holds a.e. in $K$.  Moreover, there exists a sequence $(K_n)$ of compact sets such that $\abs{ \R^N \setminus K_n} \ra 0$ as $n \ra \infty$. Therefore, either $u \equiv 0$ or $u>0$ also holds a.e. in $\R^N$.
\end{proof}

In the following proposition, we prove that $\tilde{u}$ is positive on $\R^N$ and  $u_a$ is non-negative on $\R^N$ for small enough $a$.  

\begin{proposition}\label{positivity of tilde u}
    Let $p \in (\frac{2N}{N+2s}, \frac{N}{s})$ and $\tilde{u}$ be given in Proposition \ref{properties of limit function}. Then 
    the following hold:
    \begin{enumerate}
        \item[\rm{(i)}]  $\tilde{u} > 0$ a.e. in $\R^N$.
        \item[\rm{(ii)}] Let $a_2$ be given in Proposition \ref{uniform lower bound}. Then there exists $a_3 \in (0,a_2)$ such that $u_a \ge 0$  a.e. in $\R^N$ for every $a \in (0, a_3)$. 
    \end{enumerate}
\end{proposition}

\begin{proof}
    (i) First we show that $\tilde{u}$ is non-negative on $\R^N$. Consider $A := \{ x \in \R^N : \tilde{u} \ge 0 \}$. Since $\tilde{u}$ is a weak solution of $ (-\Delta)_{p}^{s}u =  g(x)f_0(u)~ \text{in} ~\mathcal{D}^{s,p}(\mathbb{R}^{N}),$ we have 
    \begin{align*}
         \iint_{\R^{2N}} \frac{\Phi\big(\tilde{u}(x)-\tilde{u}(y)\big)\big(\phi(x) - \phi(y)\big)}{|x-y|^{N+sp}}\, \dxy &= \int_{\R^N } g(x) f_0(\tilde{u}) \phi(x) \, \dx, \\
         &= \int_{A} g(x) f_0(\tilde{u}) \phi(x) \, \dx, ~~ \forall \, \phi \in \Dsp, 
    \end{align*}
where the last identity holds since $f_0(t) = 0$ for $t \le 0$. Now we choose $\phi$ to be $-(\tilde{u})^-$ to get 
   \begin{align*}
        \iint_{\R^{2N}} \frac{\Phi\big(\tilde{u}(x)-\tilde{u}(y)\big) \left((\tilde{u})^-(y) - (\tilde{u})^-(x) \right)}{|x-y|^{N+sp}}\, \dxy = 0.
   \end{align*}
   Further, using \ref{i3}, Remark \ref{pm part}, and $\Dsp \hookrightarrow L^{p^*_s}(\R^N)$, we see that 
   \begin{align*}
       \iint_{\R^{2N}} \frac{\Phi\big(\tilde{u}(x)-\tilde{u}(y)\big) \left((\tilde{u})^-(y) - (\tilde{u})^-(x) \right)}{|x-y|^{N+sp}}\, \dxy & \ge C(p) \iint_{\R^{2N}}  \frac{\left|  (\tilde{u})^-(x) - (\tilde{u})^-(y) \right|^p}{\abs{x-y}^{N+sp}} \, \dxy \\
       & \ge C(N,s,p) \norm{(\tilde{u})^-}_{p^*_s}^p.
   \end{align*}
   Therefore, $(\tilde{u})^- = 0$ a.e. in $\R^N$ which implies that $\tilde{u} \ge 0$ a.e. in $\R^N$. Now we show the positivity of $\tilde{u}$ on $\R^N$. For a sequence $\{ a_n \}$ given in Proposition \ref{properties of limit function}, since $u_{a_n} \ra \tilde{u}$ in $L^{\infty}(\R^N)$ (Proposition \ref{unifrom convergence}) and $\norm{u_{a_n}}_{\infty} \ge C_2$ for every large enough $n$ (Proposition \ref{uniform lower bound}), there exists $C_3>0$ such that $\norm{\tilde{u}}_{\infty} \ge C_3$. Now we apply the strong maximum principle (Proposition \ref{SMP}) to conclude that  $\tilde{u} > 0$ a.e. in $\R^N$.
   
   \noi (ii) Again using $u_{a_n} \ra \tilde{u}$ in $L^{\infty}(\R^N)$ and  $\tilde{u} > 0$ a.e. in $\R^N$, there exists $n_2 \in \N$ such that $u_{a_n} \ge 0$ a.e. in $\R^N$ for all $n \ge n_2$. Thus there exists $a_3 \in (0, a_2)$ such that  for every $a \in (0, a_3)$, $u_a \ge 0$  a.e. in $\R^N$. 
\end{proof}

\begin{definition}[see \cite{Brasco2016_Optimaldeacy}]
    For an open set $\Omega \subset \R^N$, the space $\widetilde{\D}^{s,p}(\Omega)$ is defined as
\begin{align*}
    \widetilde{\D}^{s,p}(\Omega):= \Big\{ u \in L_{\rm{loc}}^{p-1}(\R^N) \cap L^{p_s^*}(\Omega): & \text{ there exists } E \supset \Omega \text{ with } E^c \text{ compact}, \text{\rm{dist}}(E^c,\Omega)>0, \no \\
    & \text{ and } \abs{u}_{W^{s,p}(E)} < \infty \Big\}.
\end{align*}
\end{definition}
It is easy to observe that $\Dsp \subset \widetilde{\D}^{s,p}(\Omega)$. Let $u \in \widetilde{\D}^{s,p}(\Omega)$. We say $(-\Delta)_p^s u = f$ weakly in $\Omega$, if 
\begin{align*}
    \iint_{\R^{2N}} \frac{\Phi\big(u(x)-u(y)\big)\big(\phi(x) - \phi(y)\big)}{|x-y|^{N+sp}}\, \dxy = \int_{\Omega} f(x) \phi(x) \, \dx, \quad \forall \, \phi \in \C_c^{\infty}(\Omega). 
\end{align*}
First, we recall the following Lemma due to \cite[Lemma A.2]{Brasco2016_Optimaldeacy}, which gives an explicit solution on the complement of a ball.
\begin{lemma}\label{Brasco Lemma A.2}
    Let $0<\frac{N-sp}{p-1}<\beta< \frac{N}{p-1}$. For every $R>0$, it holds
    \[(-\Delta)_p^s|x|^{-\beta}= C(\beta)|x|^{-\beta(p-1)-sp} \text{  weakly in } \overline{B_R}^c,\]
    where $C(\beta)$ is given by 
\begin{align}\label{3.44}
    C(\beta) = 2 \int_{0}^{1} \varrho^{sp-1}\big[1-\varrho^{N-sp-\beta(p-1)}\big] \big|1-\varrho^{\beta} \big|^{p-1} \Psi(\varrho) \, \mathrm{d} \varrho,
\end{align}
and 
\begin{align}
  \Psi(\varrho) = \mathcal{H}^{N-2}(\mathbb{S}^{N-2}) \int_{-1}^{1} \frac{(1-t^2)^{\frac{N-3}{2}}}{(1-2t\varrho+\varrho^2)^{\frac{N+sp}{2}}} \, \mathrm{d}t, 
\end{align}
where $\mathcal{H}^{N-2}$ is the Lebesgue measure of dimension $(N-2)$ and $\mathbb{S}^{N-2}$ is a unit sphere in $\mathbb{R}^{N-1}$.
\end{lemma}
It can be observed from \eqref{3.44} that $C(\beta)<0$ since $\beta>\frac{N-sp}{p-1}$. For $u, v \in \widetilde{\D}^{s,p}(\Omega)$, we say $(-\Delta)_p^s u \le (-\Delta)_p^s v$ weakly in $\Omega$, if the following holds for all $\phi \in \C_c^{\infty}(\Omega), \phi \ge 0$:
\begin{align*}
    \iint_{\R^{2N}} \frac{\Phi\big(u(x)-u(y)\big)\big(\phi(x) - \phi(y)\big)}{|x-y|^{N+sp}}\, \dxy \le \iint_{\R^{2N}} \frac{\Phi\big(v(x)-v(y)\big)\big(\phi(x) - \phi(y)\big)}{|x-y|^{N+sp}}\, \dxy.
\end{align*}
Similarly, we say $(-\Delta)_p^s u \le (\ge) f$ weakly in $\Omega$, if the following holds for all $\phi \in \C_c^{\infty}(\Omega), \phi \ge 0$:
\begin{align*}
    \iint_{\R^{2N}} \frac{\Phi\big(u(x)-u(y)\big)\big(\phi(x) - \phi(y)\big)}{|x-y|^{N+sp}}\, \dxy \le (\ge) \int_{\Omega} f(x) \phi(x) \, \dx. 
\end{align*}

Now we are ready to obtain the positivity of the solution $u_a$ for sufficiently small $a$.

\begin{theorem}\label{theorem for positivity}
    Let $p \in (\frac{2N}{N+2s}, \frac{N}{s})$. For $\frac{N-sp}{p-1}<\beta< \frac{N}{p-1}$ assume that $g$ satisfies
\begin{equation}\label{3.46}
      g(x) \leq \frac{B}{|x|^{\beta(p-1)+sp}}, \text{ for some constant } B>0 \text{ and } x\neq 0.
  \end{equation}
Let $u_a$ be given in Proposition \ref{positivity of tilde u}. Then there exists $a_4 \in (0,a_3)$ such that $u_a >0$ a.e. in $\R^N$ for every $a \in (0,a_4)$.
\end{theorem}
\begin{proof}
For $0<\varepsilon<1$, we consider the function $\Gamma(z) = \varepsilon^{\beta+1}|z|^{-\beta}$. Then for every $R>0$, $\Gamma \in \widetilde{\D}^{s,p}(\overline{B_R}^c)$ and using Lemma \ref{Brasco Lemma A.2} the following holds weakly: 
\begin{align}\label{3.47}
    (-\Delta)_p^s\Gamma(z)= \varepsilon^{(\beta+1)(p-1)}C(\beta)|z|^{-\beta(p-1)-sp} \text{ in } \overline{B_R}^c.
\end{align}
We define 
\begin{align*}
    \widetilde{\Gamma}(z) = \Gamma(z) - (\Gamma(z)-\varepsilon)_+ = \min\{\varepsilon,\Gamma(z)\},  ~z \in \R^N.
\end{align*}
Notice that $\Gamma(z) \geq \varepsilon$ if and only if $|z| \leq \varepsilon$. Thus, the support of the function $(\Gamma(z)-\varepsilon)_+$ is contained in the ball $\overline{B_\varepsilon}$. Now we choose $x \in \Omega = \overline{B_{R_1}}^c$ with $R_1 >2$, $u = \Gamma$, $f = \varepsilon^{(\beta+1)(p-1)}C(\beta)|x|^{-\beta(p-1)-sp}$, $v = - (\Gamma(x)-\epsilon)_+$ in \cite[Proposition 2.8]{Brasco2016_Optimaldeacy}. Further, $\widetilde{\Gamma} \in \widetilde{\D}^{s,p}(\overline{B_{R_1}}^c)$.  Hence, in view of \eqref{3.47} with $R=R_1$, the following holds weakly in $\overline{B_{R_1}}^c$:
\begin{align} \label{3.48}
    (-\Delta)_p^s  \widetilde{\Gamma}(x) &= \varepsilon^{(\beta+1)(p-1)}C(\beta)|x|^{-\beta(p-1)-sp} + 2 \int_{B_\varepsilon} \frac{\Phi(\Gamma(x)-\varepsilon)- \Phi(\Gamma(x)-\Gamma(y))}{|x-y|^{N+sp}}\,\dy \no\\
    & = \varepsilon^{(\beta+1)(p-1)}C(\beta)|x|^{-\beta(p-1)-sp} + 2 \int_{B_\varepsilon} \frac{\Phi(\Gamma(y)-\Gamma(x))-\Phi(\varepsilon -\Gamma(x))}{|x-y|^{N+sp}}\,\dy.
\end{align}
Since $|x|>R_1 >2$ and $|y|\leq \varepsilon$, it easily follows that $\Gamma(x)<\Gamma(y)$ and $\Gamma(x)<\varepsilon$. Thus, we have the following estimate
\begin{align*}
    \Phi(\Gamma(y)-\Gamma(x)) - \Phi(\epsilon - \Gamma(x)) = (\Gamma(y)-\Gamma(x))^{p-1} - (\epsilon - \Gamma(x))^{p-1} \leq (\Gamma(y)-\Gamma(x))^{p-1} \leq(\Gamma(y))^{p-1}.
\end{align*}
Further,
\begin{align*}
|x-y| \geq |x|-|y| \geq |x|-\varepsilon \geq |x| - \frac{|x|}{2} = \frac{|x|}{2}.    
\end{align*}
Using the above two estimates in \eqref{3.48}, the following holds weakly in $\overline{B_{R_1}}^c$:
\begin{align*}
    (-\Delta)_p^s  \widetilde{\Gamma}(x) \leq \varepsilon^{(\beta+1)(p-1)} \frac{C(\beta)}{|x|^{\beta(p-1)+sp}} + \frac{2^{N+sp+1}}{|x|^{N+sp}} \int_{B_\varepsilon} (\Gamma(y))^{p-1}\,\dy,
\end{align*}
where we calculate
\begin{align*}
 \int_{B_\varepsilon} (\Gamma(y))^{p-1}\,\dy = \varepsilon^{(\beta+1)(p-1)} \int_{B_\varepsilon} |y|^{-\beta(p-1)}\,\dy & = \sigma(\mathbb{S}^{N-1}) \varepsilon^{(\beta+1)(p-1)} \int_{0}^{\varepsilon} r^{N-\beta(p-1)-1}\,\mathrm{d}r \\
  & = \sigma(\mathbb{S}^{N-1}) \varepsilon^{(\beta+1)(p-1)} \frac{ {\varepsilon}^{N-\beta(p-1)} }{N-\beta(p-1)},
\end{align*}
where the quantity $\sigma(\mathbb{S}^{N-1})$ denotes the $(N-1)$-dimensional measure of the unit sphere in $\R^N$. Therefore, for any $0< \varepsilon < 1$, the following holds weakly in $\overline{B_{R_1}}^c$:
\begin{align*}
    (-\Delta)_p^s  \widetilde{\Gamma}(x) & \leq \left( \frac{C(\beta)}{|x|^{\beta(p-1)+sp}} + \frac{2^{N+sp+1} \sigma(\mathbb{S}^{N-1}) {\varepsilon}^{N-\beta(p-1)} }{|x|^{N+sp}(N-\beta(p-1))} \right) \varepsilon^{(\beta+1)(p-1)} \\
    &\leq  \left( C(\beta) + \frac{2^{N+sp+1} \sigma(\mathbb{S}^{N-1}) {\varepsilon}^{N-\beta(p-1)}}{(N-\beta(p-1))} \right) \varepsilon^{(\beta+1)(p-1)}|x|^{-\beta(p-1)-sp}\\
    &:= C_1(\beta,\varepsilon) \varepsilon^{(\beta+1)(p-1)}|x|^{-\beta(p-1)-sp}.
\end{align*}
The second last inequality follows from the fact that $|x|>1$ and $\beta(p-1)+sp < N+sp$. Using the fact that $C(\be)<0$, now we choose $0< \varepsilon <1$ small enough so that 
\begin{align}\label{3.49}
    {\varepsilon} < \left( \frac{-C(\beta)(N-\beta(p-1))}{ 2^{N+sp+1} \sigma(\mathbb{S}^{N-1}) } \right)^{\frac{1}{N-\beta(p-1)}}.
\end{align}
Therefore, for $\varepsilon$ as in \eqref{3.49} the following holds weakly 
\begin{align} \label{3.50}
     (-\Delta)_p^s \widetilde{\Gamma}(x) \leq C_1(\beta,\varepsilon) {\varepsilon}^{(\beta+1)(p-1)}|x|^{-\beta(p-1)-sp} \text{ in } \overline{B_{R_1}}^c,
\end{align}
where $C_1(\beta,\varepsilon)<0$.
Since $a_n \rightarrow 0$, there exists $n_1 \in \N$ such that $a_n B \leq -C_1(\beta,\varepsilon){\varepsilon}^{(\beta+1)(p-1)}$ and $u_{a_n} \ge 0$ a.e. in $\R^N$ (Proposition \ref{positivity of tilde u}) for all $n \geq n_1$. Thus, for all $n \geq n_1$ using the assumption \eqref{3.46}, the following holds weakly in $\overline{B_{R_1}}^c$:
\begin{align} \label{3.51}
     (-\Delta)_p^s u_{a_n} &= g(x) (f(u_{a_n})-a_n) \geq -a_n g(x)  \no \\
     &\geq - a_n B |x|^{-\beta(p-1)-sp} \geq C_1(\beta,\varepsilon) {\varepsilon}^{(\beta+1)(p-1)}|x|^{-\beta(p-1)-sp}.
\end{align}
Further, $u_{a_n} \in \widetilde{\D}^{s,p}(\overline{B_{R_1}}^c)$. Combining \eqref{3.50} and \eqref{3.51}, for all $n \geq n_1$, the following holds weakly:
\begin{align} \label{3.52}
    (-\Delta)_p^s \widetilde{\Gamma} \leq (-\Delta)_p^s u_{a_n}  \text{ in } \overline{B_{R_1}}^c.
\end{align}
Since $u_{a_n} \rightarrow \tilde{u}$ in $L^\infty(\R^N)$ (Proposition \ref{unifrom convergence}) and $\tilde{u}>0$ in $\R^N$ (Proposition \ref{positivity of tilde u}), there exists $\eta_1>0$ depending on $R_1$ and $n_2 \in \N$ such that $u_{a_n} \geq \eta_1$ on $\overline{B_{R_1}}$ for all $n \geq n_2$. For $0<\varepsilon <1$ as in \eqref{3.49} we further choose $\varepsilon < \eta_1$. Hence
\begin{align} \label{3.53}
    u_{a_n} \geq \eta_1 > \varepsilon \geq \widetilde{\Gamma} \text{ in }  \overline{B_{R_1}}, \quad \forall \, n \geq n_2.
\end{align}
Using \eqref{3.52}, \eqref{3.53}, and applying the comparison principle \cite[Theorem 2.7]{Brasco2016_Optimaldeacy}, we obtain
\begin{align}\label{3.54}
    u_{a_n} \geq \widetilde{\Gamma} \text{ in }  \overline{B_{R_1}}^c, \quad \forall \, n \geq n_3:= \max\{n_1,n_2\}.
\end{align}
Thus, we deduce from \eqref{3.53} and \eqref{3.54} that $ u_{a_n}(x) \geq \widetilde{\Gamma}(x) \text{ a.e. in } \R^N$ for all $n \geq n_3$. As a result, there exists $a_4 \in (0,a_3)$ such that $u_a >0$ a.e. in $\R^N$ for all $a \in (0,a_4)$. This completes the proof. 
\end{proof}

\noi \textbf{Proof of Theorem \ref{Theorem 1.1}:}
The proof of part $\rm{(a)}$ follows by Theorem \ref{Existence and uniform boundedness}. Part $\rm{(b)}$ is a consequence of Proposition \ref{regularity}. The proof of part $\rm{(c)}$ is provided in Proposition \ref{positivity of tilde u}, whereas the positivity of solutions is demonstrated in Theorem \ref{theorem for positivity}, which corresponds to part $\rm{(d)}$. 
\qed
\begin{example}
  Let $N>sp$. For some given constants $A,B>0$ and $\gamma \in (p,p_s^*)$, we consider the following functions:
  \begin{align*}
      f(t)= At^{\gamma -1}, \text{ for }t \in \R^+ \text{ and }   g(x)= \frac{B}{1+|x|^{\beta(p-1)+sp}}, \text{ for } x \in \R^N.
  \end{align*}   
$\rm{(i)}$ It is evident that the function $f$ satisfies {\rm{\ref{f1}}}, {\rm\ref{f2}}, and {\rm\ref{f1 new}}.\\
$\rm{(ii)}$ Clearly, $g \in L^\infty(\R^N)$ and satisfies \eqref{g1 new1}. Now we show that $g \in L^1(\R^N)$.
\begin{align*}
    \int_{\R^N} g(x)\,\dx &= \int_{\R^N}\frac{B}{1+|x|^{\beta(p-1)+sp}} \,\dx = B \sigma(\mathbb{S}^{N-1}) \bigg( \int_{0}^{1} + \int_{1}^{\infty} \bigg)\frac{t^{N-1}}{1+t^{\beta(p-1)+sp}} \,\mathrm{d}t\\
    &\leq B \sigma(\mathbb{S}^{N-1}) \bigg(1 + \int_{1}^{\infty}t^{N-1 -\beta(p-1)-sp} \,\mathrm{d} t \bigg) \leq C(N,s,p,\beta,B).
\end{align*}
The last integral is finite since $\beta(p-1)>N-sp$.
\end{example}

\noi \textbf{Acknowledgements:}
N. Biswas is funded by the Department of Atomic Energy, Government of India, under project no. 12-R$\&$D-TFR-5.01-0520 and the Science and Engineering Research Board, Government of India, through the National Postdoctoral Fellowship (PDF/2023/000038). R. Kumar acknowledges the support of the CSIR fellowship, file no. 09/1125(0016)/2020--EMR--I. The authors thank Prof. Abhishek Sarkar (IIT Jodhpur) for the beneficial discussions. 

\noi \textbf{Declarations:}
The authors declare that there is no conflict of interest. Data sharing does not apply to this article, as no data sets were generated or analyzed during the current study.
\appendix
\section{}
This section contains some technical results.

\begin{lemma}
Let $p \in (1,\infty), ~\beta \geq 1, a,b \in \R$ and $M > 0$. Let $C(\beta,p) = \big(\frac{p}{\beta + p -1} \big)^p$. Then the following inequalities hold:
\begin{enumerate}[label={\rm{($\bf i{\arabic*}$)}}]
\setcounter{enumi}{0}
    \item \label{i1}  $|a-b|^{p-2}(a-b)\left( (a^+)^\beta - (b^+)^\beta \right) \geq C(\beta,p) \left|(a^+)^{\frac{\beta +p-1}{p}} - (b^+)^{\frac{\beta + p-1}{p}} \right|^p$ where $a^+ = \max\{a,0\}$.
    \item \label{i2} $|a-b|^{p-2}(a-b) \left( (a^+)_M^\beta - (b^+)_M^\beta \right) \geq C(\beta,p) \left|(a^+)_M^{\frac{\beta +p-1}{p}} - (b^+)_M^{\frac{\beta + p-1}{p}} \right|^p$ where $(a^+)_M = \min \{ a^+, M\}$ and $(b^+)_M = \min \{b^+, M\}$.
    \item \label{i3} $|a-b|^{p-2}(a-b)\left( (b^-)^\beta - (a^-)^\beta \right) \geq C(\beta,p) \left|(a^-)^{\frac{\beta +p-1}{p}} - (b^-)^{\frac{\beta + p-1}{p}} \right|^p$ where $a^- = -\min\{a,0\}$.
    \item \label{i4} $|a-b|^{p-2}(a-b) \left( (b^-)_M^\beta - (a^-)_M^\beta \right) \geq C(\beta,p) \left|(a^-)_M^{\frac{\beta +p-1}{p}} - (b^-)_M^{\frac{\beta + p-1}{p}} \right|^p$ where $(a^-)_M = \min \{ a^-, M\}$ and $(b^-)_M = \min \{b^-, M\}$.
\end{enumerate}

\end{lemma}
\begin{proof}
(i) If $a=b$, then \ref{i1} holds trivially. So we assume that $a \neq b$. Without loss of generality, we can assume that $a>b$. For $a,b>0$, $a^+ =a, b^+=b$, and \ref{i1} follows using \cite[Lemma C.1]{Brasco2014_fractional_Cheeger}. If $a,b<0$ then \ref{i1} holds trivially since $a^+=0=b^+$. Next, we assume $a>0>b$. Since $a^+ = a$ and $b^+ =0$, we need to show that
\begin{equation}\label{A.1}
    (a-b)^{p-1} \geq C(\beta,p) a^{p-1}.
\end{equation}
If we divide \eqref{A.1} by $a^{p-1}$, we get $(1-\frac{b}{a})^{p-1} \geq C(\beta,p)$ and this inequality always holds true since $C(\beta,p) \leq 1$ and $\frac{b}{a}<0$. 

\noi (ii) Now we consider $M > 0$. For $a=b$, \ref{i2} holds trivially. So without loss of generality, we assume that $a>b$. If $a,b \ge M$ then $(a^+)_M = M = (b^+)_M$, and \ref{i2} holds trivially. If $a, b \le M$, then by noting that $(a^+)_M = a^+,  (b^+)_M = b^+$, \ref{i2}  follows using \ref{i1}. Now we assume that $b<M<a$. In this case $(b^+)_M = b^+ < M = (a^+)_M$. Hence using \ref{i1} we get
\begin{align*}
    |a-b|^{p-2}(a-b)\big( (a^+)_{M}^\beta - (b^+)_{M}^\beta \big) & \geq |M-b|^{p-2} (M-b) \left(M^{\beta} - (b^+)^{\beta } \right) \\
    &\ge C(\beta,p) \left|M^{\frac{\beta +p-1}{p}} - (b^+)^{\frac{\beta + p-1}{p}} \right|^p \\
    &= C(\beta,p) \left|(a^+)_M^{\frac{\beta +p-1}{p}} - (b^+)_M^{\frac{\beta + p-1}{p}} \right|^p.
\end{align*}
Thus \ref{i2} holds for every $a, b \in \R$.

\noi (iii) Without loss of generality, assume that $a>b$. If $a,b>0$ then \ref{i3} holds trivially since $a^-=0=b^-$. For $a,b<0$, $a^- =-a, b^-=-b$, where $0 \le a^-<b^-$. Applying \cite[Lemma C.1]{Brasco2014_fractional_Cheeger} we get 
\begin{align*}
    |a^- - b^-|^{p-2}(a^- - b^-)\left( (a^-)^\beta - (b^-)^\beta \right) \geq C(\beta,p) \left|(a^-)^{\frac{\beta +p-1}{p}} - (b^-)^{\frac{\beta + p-1}{p}} \right|^p.
\end{align*}
The above inequality infers \ref{i3}. Next, we consider $a>0>b$. In this case, \ref{i3} has the following form: 
\begin{align*}
    (a-b)^{p-1} \ge C( \be, p) (-b)^{p-1}. 
\end{align*}
Dividing the above inequality by $(-b)^{p-1}$, we see $\left(-\frac{a}{b} + 1 \right)^{p-1} \ge C( \be, p)$. This inequality always holds true since $-\frac{a}{b} \ge 0$ and $C(\be, p) \le 1$. \\
\noi (iv) Now we consider $M>0$. For $a=b$, \ref{i4} holds trivially. So without loss of generality, we assume that $a>b$.  For $a,b \ge 0$, \ref{i4} trivially holds. If $a,b \leq -M$, then $(a^-)_M = M = (b^-)_M$, and \ref{i4} holds. If $a,b \geq -M$, then by noticing that $(a^-)_M = a^-, (b^-)_M=b^-$, \ref{i4} follows using \ref{i3}. Now we assume that $b<-M<a$. For the case $a,b<0$, we notice that $(a^-)_M = a^-, (b^-)_M=M$ and consequently, using \ref{i3} we obtain
\begin{align*}
    |a-b|^{p-2}(a-b)\big((b^-)_{M}^\beta - (a^-)_{M}^\beta \big) &\geq |a-(-M)|^{p-2}(a-(-M))\big((M)^\beta - (a^-)^\beta \big)\\
    & \geq C(\beta,p) \bigg|(a^-)^{\frac{\beta +p-1}{p}} - M^{\frac{\beta +p-1}{p}} \bigg|^p\\
    &= C(\beta,p) \bigg|(a^-)_M^{\frac{\beta +p-1}{p}} - (b^-)_M^{\frac{\beta +p-1}{p}} \bigg|^p.
\end{align*}
Now if we consider $a>0>b$, then $(a^-)_M=0, (b^-)_M = M$. In this case,
\begin{align*}
    \text{\ref{i4}} \Longleftrightarrow (a-b)^{p-1} (b^-)_M^\beta \geq C(\beta,p) (b^-)_M^{\beta +p-1} 
    \Longleftrightarrow \bigg(\frac{a-b}{M}\bigg)^{p-1} \geq C(\beta,p).
\end{align*}
The last inequality holds since $C(\beta,p) \leq 1$ and $\frac{a}{M}- \frac{b}{M} > \frac{a}{M}+ 1>1$. Thus, \ref{i4} holds for every $a,b \in \R$.
\end{proof}
\begin{lemma} \label{appendix lem1}
   An iterative sequence is defined as
   \begin{align}\label{4.1}
       l_0 = p_s^*,  \quad l_{j+1}= \theta l_j + \frac{p^*_s(p-1)}{p}, \text{ and } \theta = \frac{p^*_s}{p \Bar{r}'},
   \end{align}
 where $\Bar{r}'=\frac{\Bar{r}}{\Bar{r}-1}$ and $\Bar{r}$ is given by \eqref{3.25}. Then there exist $\beta_1,\beta_2>0$ such that 
$\beta_1<\frac{\theta^j}{l_j}<\beta_2$ for all $j \in \N$.
\end{lemma}
\begin{proof}
    From \eqref{4.1}, we can write the iterative sequence as
    \begin{align} \label{4.2}
        l_j = \theta^j l_0 + \frac{p_s^*(p-1)}{p} \sum\limits_{i=0}^{j-1} \theta^{i} = \theta^j p_s^* + \frac{p_s^*(p-1)}{p} \sum\limits_{i=0}^{j-1} \theta^{i}.
    \end{align}
Dividing \eqref{4.2} by $\theta^j$, we obtain
\begin{align*}
      \frac{l_j}{\theta^j} &= p_s^* + \frac{p_s^*(p-1)}{p} \left(\frac{1}{\theta^j} + \frac{1}{\theta^{j-1}}+ \cdots +\frac{1}{\theta} \right) = p_s^* + \frac{p_s^*(p-1)}{p} \sum\limits_{i=1}^{j} \frac{1}{\theta^i}< p_s^* + \frac{p_s^*(p-1)}{p} \sum\limits_{i=1}^{\infty} \frac{1}{\theta^i}.
 \end{align*}
Consequently, 
\begin{align}\label{4.3}
\frac{l_j}{\theta^j} < p_s^* + \frac{p_s^*(p-1)}{p} \frac{1}{\theta-1}.
\end{align}
Furthermore, it is evident from \eqref{4.2} that
\begin{align} \label{4.4}
\frac{l_j}{\theta^j} > p_s^*.
\end{align}
The combination of \eqref{4.3} and \eqref{4.4} yields
$\beta_1<\frac{\theta^j}{l_j} < \beta_2,$ for all $j \in \N$,
where $\beta_1= \frac{1}{p_s^* + \frac{p_s^*(p-1)}{p(\theta-1)}}$ and $\beta_2 = \frac{1}{p_s^*}$.   
\end{proof}
\bibliographystyle{abbrv}

\end{document}